\documentclass[11pt]{amsart}
\usepackage{amsmath}
\usepackage{tikz}

\theoremstyle{plain}
\newtheorem{theorem}{Theorem}[section]
\newtheorem{lemma}[theorem]{Lemma}

\newtheorem{corollary}[theorem]{Corollary}

\theoremstyle{definition}
\newtheorem{remark}[theorem]{Remark}

\numberwithin{equation}{section}

\def\be{\begin{equation}}
\def\ee{\end{equation}}

\begin{document}

\title[Asymptotics for  Null-timelike Boundary Problems]
{Asymptotics for  Null-timelike Boundary Problems for General Linear Wave Equations}
\author[Han]{Qing Han}
\address{Department of Mathematics\\
University of Notre Dame\\
Notre Dame, IN 46556} \email{qhan@nd.edu}
\address{Beijing International Center for Mathematical Research\\
Peking University\\
Beijing, 100871, China} \email{qhan@math.pku.edu.cn}
\author[Zhang]{Lin Zhang}
\address{School of Mathematical Sciences\\
Peking University\\
Beijing, 100871, China}
\email{linzhang2013@pku.edu.cn}

\begin{abstract}
We study the linear wave equation $\Box_{g}u=0$ in Bondi-Sachs coordinates,
for  an asymptotically flat Lorentz metric $g$. We consider the null-timelike boundary problem,
where an initial value is given on the null surface $\tau=0$ and a boundary value on the timelike surface $r=r_{0}$.
We obtain spacetime $H^{p}$-estimates of $ru$ for $r>r_0$ and
derive an asymptotic exapnsion of $ru$ in terms of ${1}/{r}$ as $r\to\infty$.
\end{abstract}

\thanks{The first author acknowledges the support of NSF
Grant DMS-1404596. The second author acknowledges the support of NSFC
Grant 11571019.}
\maketitle

\section{Introduction}\label{sec-Introduction}

In this paper, we will study the asymptotic behaviors of solutions of the null-timelike boundary problem
for the general wave equations.
In the simplest setting, consider a $C^2$-function $u=u(x,t)$, for $x\in \mathbb R^3$,
satisfying the wave equation
\begin{equation}\label{eq-wave-equation}\partial_{tt}u-\Delta u=0,\end{equation}
in $\{(x,t):\, R<|x|<t\}$, for some $R>0$, and $u=0$ on $t=|x|>R$.
Friedlander \cite{Friedlander1962}-\cite{Friedlander1967} discussed in a more general setting and proved
among other results, that
\begin{equation}\label{eq-definition-radiation-field}
|x|u(x,t)\sim v_0(x/|x|, t-|x|)\quad\text{for large }|x|,\end{equation}
where $v_0$ is a function of the direction $x/|x|\in  S^2$ and of the {\it retarded time} $\tau=t-|x|$,
and is called the {\it radiation field} of $u$.
In terms of spherical coordinates, with $r=|x|$, and the retarded time $\tau$, the equation \eqref{eq-wave-equation}
can be written as
\begin{equation}\label{eq-wave-equation-alternative}-2\partial_{\tau r}(ru)+\partial_{rr}(ru)+\partial_r(ru)
+\frac{1}{r^2}\Delta_{S^2} (ru)=0,\end{equation}
where $\Delta_{S^2}$ is the Beltrami-Laplace operator on $S^2$.

In this paper, we will discuss radiation fields for general asymptotically flat metrics $g$.
Our primary goal is to derive an estimate of expansions in terms of given data
in suitable Sobolev norms.
To achieve these goals,
we will adopt the Bondi-Sachs coordinates
$\{\tau, r, x_{2}, x_{3}\}$, with $\tau\in (0,T)$, $r\in (R,\infty)$, and
$(x_2, x_3)$ forming local coordinates on $S^2$, for some fixed constants $T$ and $R$.
For example, in such coordinates,  Minkowski metric $g_M$ is given by
\begin{equation*}
g_M
=-d\tau^{2}-2d\tau dr+r^{2}g_{S^2}. \end{equation*}
Let $g$ be a sufficiently smooth asymptotically flat Lorentz metric in Bondi-Sachs coordinates.
(Refer to Section \ref{sec-Metrics}
for a precise formulation.)
We consider the wave equation
\begin{equation}\label{eq-wave-equation-general}\Box_g u=0\quad\text{in }(0,T)\times (R, \infty)\times S^2,\end{equation}
together with
\begin{align}\label{eq-wave-condition-general}\begin{split}
u|_{\tau=0}&=\frac{\varphi}{r} \quad\text{on }[R,\infty)\times S^2,\\
u|_{r=R}&=\psi \quad\text{on }[0, T)\times S^2,
\end{split}\end{align}
for some given functions $\varphi$ on $[R,\infty)\times S^2$ and
$\psi$ on $[0, T)\times S^2$.
Here, $\varphi/r$ and $\psi$ are considered to be the initial and boundary values, respectively.
For the function $\varphi$ given on  $(R, \infty)\times S^2$, we define
$$\|\varphi\|_{\widetilde H^{p}((R, \infty)\times S^2)}=
\sum_{|\alpha|\leq p,\alpha=(\alpha_1,\alpha_2,\alpha_3)}\big(\int_R^{\infty}\int_{{S}^{2}}
\frac{|\partial^{\alpha}f|^{2}}{r^{2+4\alpha_1}}drd\Sigma\big)^{\frac{1}{2}}.$$
We say $\varphi\in \widetilde H^{p}((R, \infty)\times S^2)$ if
$\|\varphi\|_{\widetilde H^{p}((R, \infty)\times S^2)}<\infty$.

Our primary goal is to study behaviors of $ru$ as $r\to\infty$. We will prove the following result.

\begin{theorem}\label{thrm-main} For some integer $k\ge 1$,
let $g$ be a $C^{k+2}$ metric given by \eqref{eq-definition-g}
on $[0,T]\times [R,\infty)\times S^2$ and satisfy \eqref{afc-intro} and \eqref{eq-metric-h-intro}.
Suppose $\varphi\in \widetilde H^{2k+5}((R, \infty)\times S^2)$ and
$\psi\in H^{2k+5}((0,T)\times S^2)$, with $\varphi(R,\cdot)=\psi(0,\cdot)$ on $S^2$.
Then, there exists a unique solution $u\in C^{k}([0,T]\times [R, \infty)\times S^2)$
of \eqref{eq-wave-equation-general}-\eqref{eq-wave-condition-general}.
Moreover, there exists  $v_i\in C^{k-i}([0, T]\times S^2)$, $i=0,\cdot\cdot\cdot,k-1$, such that,
\begin{equation}\label{eq-main-estimate}
\left|ru-\sum_{i=0}^{k-1}\frac{v_i}{r^i}\right|\le \frac{C}{r^{k}}\left\{\|\varphi\|_{\widetilde H^{2k+5}((R, \infty)\times S^2)}
+\|\psi\|_{H^{2k+5}((0,T)\times S^2)}\right\},\end{equation}
and
\begin{equation}\label{eq-main-estimate-1}
\|v_i\|_{C^{k-i}([0, T]\times S^2)}\le C\left\{\|\varphi\|_{\widetilde H^{2k+5}((R, \infty)\times S^2)}
+\|\psi\|_{ H^{2k+5}((0,T)\times S^2)}\right\},\end{equation}
where $C$ is a positive constant depending only on $k$, $T$, $R$ and  $g$.
\end{theorem}

The function $v_0$ is the {\it radiation field} of $u$.
Once $v_0$ is known, the rest of $v_i$'s can be computed explicitly, in terms of $\varphi$, $\psi$, $v_0$
and the metric $g$. The estimate \eqref{eq-main-estimate-1} illustrates that there is a loss of differentiation
by a {\it half}. This is typical for the null-timelike boundary problem,
or the characteristic initial value problem, as
\eqref{eq-wave-equation-general}-\eqref{eq-wave-condition-general}.

The radiation field, or more generally, the expansion \eqref{eq-main-estimate},
describes asymptotic behaviors of solutions as $r\to\infty$. The study of the radiation fields is closely related to
that of the gravitational waves as in \cite{Bondi1962} and \cite{Sachs1962}.
In this paper, following Friedlander \cite{Friedlander1962}-\cite{Friedlander1967},
we study such expansions for the null-timelike boundary problem and provide a relatively precise
estimate. Radiation fields for the initial-value problem is also intensively studied.
The existence of radiation fields and asymptotic behaviors are discussed in
a various of settings,  on asymptotically Euclidean manifolds \cite{Baskin2}, \cite{Baskin1},
\cite{Friedlander1980}, \cite{Friedlander2001}, \cite{A6},
on asymptotically hyperbolic spaces \cite{A3},
on the Schwarzschild space \cite{Wang}, on the de Sitter-Schwarzschild space \cite{Melrose},
and, in a nonlinear setting, for the Einstein vacuum equations \cite{Wang2}.
Properties of radiation fields were studied in \cite{Baskin3}, \cite{A2}, \cite{A}, \cite{A4}.
The study of radiation fields is also related to that of the Bondi mass and Bondi-Sachs metrics
\cite{ZX2}, \cite{ZX}, \cite{ZX3}.

A significant portion of this paper is devoted to a derivation of energy estimates associated with the problem
\eqref{eq-wave-equation-general} and \eqref{eq-wave-condition-general}. With these energy estimates, we can
obtain the existence of solutions of \eqref{eq-wave-equation-general} and \eqref{eq-wave-condition-general}
and the existence of functions $v_i$ together with the estimate \eqref{eq-main-estimate}
and \eqref{eq-main-estimate-1}.

In the classical energy estimates for the initial/boundary value problem of the wave equation, the $H^1$-norm of
solutions in domains is estimated in terms of the $H^1$-norm of the initial values and the $L^2$-norm of the boundary
values. This permits us to iterate such estimates for higher order derivatives. However, in our case,
the $H^1$-norm of
solutions in domains is estimated in terms of the $H^1$-norm of the initial values and also the $H^1$-norm of the boundary
values. (See Theorem \ref{maintheo}.)
The presence of the $H^1$-norm of the boundary values causes difficulties when we attempt to apply such estimates
to higher order derivatives.  Derivatives of solutions along the normal direction
restricted to boundary are not controlled by known quantities.
In fact, normal derivatives of a certain order can be estimated in terms of derivatives
in domains of the same order. This will form a cycle when we attempt to estimate derivatives.
Specifically, a collection of derivatives of some order in the domain can be estimated in terms of normal
derivatives of the same order on boundary, which in turn can be estimated in terms of
another collection of derivatives of the same order in the domain, which is yet
to be controlled. To break this cycle, we need to form a
{\it closed} system of equations  in the sense that
the above mentioned two collections of derivatives of the same order are identical.
For the characteristic initial-value problem such as
\eqref{eq-wave-equation-general}-\eqref{eq-wave-condition-general},
there is also a loss of differentiation along the time direction.
We need to find a closed system to minimize such a loss.

For the Minkowski metric, the system for derivatives with respect to only spherical coordinates
forms a closed system,
and the system for derivatives with respect to spherical and radial coordinates
also forms a closed system.
In this case, we can estimate derivatives with respect to only spherical coordinates first,
then  derivatives with respect to spherical and radial coordinates, and last arbitrary derivatives.
Refer to \cite{Balean1997}.
For the general metric, the system for derivatives with respect to only spherical coordinates
does not necessarily form a closed system. However, the system for derivatives with respect
to spherical and radial coordinates forms a closed system. We can start with this system
and derive estimates of derivatives  with respect
to spherical and radial coordinates. There is no loss of differentiation at this step. Then, we extend
our estimates to derivatives with respect time.
Usually, there is a loss of regularity by {\it half} along the time direction. Refer to \cite{Hagen1977}
for more details.

The paper is organized as follows. In Section \ref{sec-Metrics}, we rewrite the metric $g$
in \eqref{eq-definition-g} and the wave equation \eqref{eq-wave-equation-general} in a bounded domain.
In Section \ref{sec-H1-Estimates} and Section \ref{sec-Hk-Estimates},
we derive energy estimates for the newly formulated wave equation.
These two sections form the main part of the paper.
With these estimates, we prove the existence of solutions and the existence of radiation fields
in Section \ref{sec-Existence}.

We would like to thank Xiao Zhang for many helpful discussions.

\section{Metrics}\label{sec-Metrics}

In this section, we present a precise formulation of metrics $g$ introduced in the introduction and
represent them in a bounded domain.

Let $g$ be a Lotentz metric.
In the following, we adopt the Bondi-Sachs coordinates
$\{\tau, r, x_{2}, x_{3}\}$, with $\tau\in (0,T)$, $r>R$, and
$(x_2, x_3)$ forming local coordinates on $S^2$, for some fixed constants $T$ and $R$.
Refer to \cite{Bondi1962} or \cite{Sachs1962} for details.
Suppose $g$ is given by
\begin{equation}\label{eq-definition-g}
g
=-Ve^{2\eta}d\tau^{2}-2e^{2\eta}d\tau dr+r^{2}h_{AB}(dx^{A}-U^{A}d\tau)(dx^{B}-U^{B}d\tau),
\end{equation}
where $V, \eta, U^{A}, h_{AB}$ are functions of $\tau, r, x^A$, with $A, B=2, 3$, and that $\det(h_{AB})$ is independent of $\tau$ and $r$.
As in \cite{Sachs1962}, we assume that $g$ is asymptotically flat, i.e., as $r\to\infty$,
\begin{equation}\label{afc-intro}
V\to 1,\quad
rU^{A}\to 0, \quad \eta\to 0,
\end{equation}
and
\begin{align}\label{eq-metric-h-intro}h\equiv h_{AB}dx^Adx^B \rightarrow g_{S^2},\end{align}
uniformly on $[0,T]\times S^2$, where $g_{S^2}$ is the standard round metric on $S^2$.
We also assume that derivatives  of $V$, $\eta$, $U^{A}$, $h_{AB}$ up to a certain order
with respect to $\tau, z=1/r, x^A$ are bounded.

The Minkowski metric is given by \eqref{eq-definition-g} with
$$V=1,\ \eta=0,\ U^A=0, \ h = g_{S^2},$$
and the Schwarzschild metric with
$$V=1-\frac{2M}{r},\ \eta=0,\ U^A=0, \ h = g_{S^2},$$
for some positive constant $M$.

Our primary goal in this paper is to study the equation
\begin{equation}\label{eq-equation-1}
\Box_g u=0,\end{equation}
and analyze behaviors of $ru$ as $r\to\infty$.

Following \cite{Friedlander1962} and \cite{Friedlander1967}, we introduce
a change of variables
$$z=\frac{1}{r}.$$
In the following, we denote the new coordinates by $(x^0, x^1, x^2, x^3)$, with $x^0=\tau$ and $x^1=z$.
Then, in these coordinates,
\begin{equation*}
g=-Ve^{2\eta}d\tau^{2}+\frac{2}{z^{2}}e^{2\eta}d\tau dz
+\frac{1}{z^{2}}h_{AB}(dx^{A}-U^{A}d\tau)(dx^{B}-U^{B}d\tau).
\end{equation*}
The associated wave operator is defined by
\begin{align*}
\Box_{{g}}u=\frac{1}{(-\det {g})^{{1}/{2}}}\sum\limits_{i=0}^{3}\partial_{i}[(-\det {g})^{{1}/{2}}{g}^{ij}\partial_ju].
\end{align*}
A straightforward calculation yields
\begin{equation*}
{g}^{ij}=\begin{bmatrix}
0& z^{2}e^{-2\eta}&0 &0\\
z^{2}e^{-2\eta}&z^{4}Ve^{-2\eta}&z^{2}U^{2}e^{-2\eta}&z^{2}U^{3}e^{-2\eta}\\
0&z^{2}U^{2}e^{-2\eta}&z^{2}h^{22}&z^{2}h^{23}\\
0&z^{2}U^{3}e^{-2\eta}&z^{2}h^{32}&z^{2}h^{33}
\end{bmatrix},
\end{equation*}
where $h^{AB}h_{BC}=\delta^{A}_{C}$.
As in \cite{Friedlander1967},  we introduce the conformal metric $\hat{g}$ given by
\begin{equation*}
\hat{g}=z^{2}{g}=-z^2Ve^{2\eta}d\tau^{2}+2e^{2\eta}d\tau dz
+h_{AB}(dx^{A}-U^{A}d\tau)(dx^{B}-U^{B}d\tau).
\end{equation*}
Set
\begin{equation}\label{eq-definition-v}v=ru.\end{equation}
Then,
\begin{equation*}
\Box_{g}u=z^{3}\Box_{\hat{g}}v+v\Box_{g}z.
\end{equation*}
Note
\begin{align*}
\Box_{g}z&=\frac{1}{(-\det {g})^{{1}/{2}}}\sum\limits_{i=0}^{3}\partial_{i}[(-\det {g})^{{1}/{2}}{g}^{i1}]\\
&=\frac{z^{4}}{e^{2\eta}\kappa}\{z^{-2}\partial_{\tau}\kappa-\partial_{z}(\kappa V)+z^{-2}[\partial_{x_{A}}(\kappa U^{A})]\},
\end{align*}
where $\kappa=[\det(h_{AB})]^{{1}/{2}}$.
By \eqref{afc-intro}, in particular $\lim\limits_{r\rightarrow\infty}(rU^{A})=0$, and since $\kappa$ is independent of $\tau$, we write
$$\Box_{g}z=e^{-2\eta}z^{3}\omega,$$
where $\omega$ is a sufficiently smooth function.
The equation \eqref{eq-equation-1} reduces to the following equation for $v$:
\begin{equation*}
e^{2\eta}\Box_{\hat{g}}v+\omega v=0.
\end{equation*}
Next, we introduce another conformal metric
\begin{equation}\label{eq-metric-g-bar}\bar{g}=e^{-2\eta}\hat{g}=-z^2Vd\tau^{2}+2d\tau dz
+e^{-2\eta}h_{AB}(dx^{A}-U^{A}d\tau)(dx^{B}-U^{B}d\tau).
\end{equation}
Then, $v$ satisfies
\begin{equation}\label{eq-quation-main-form}
\Box_{\bar{g}}v+a^{i}\partial_{i}v+\omega v=0,
\end{equation}
where
$a^{i}$ are functions expressed in terms of $V$, $U^{A}$, ${g}^{ij}$, $\eta$, $\partial_{i}\eta$, for $i=1, 2, 3$.

The equation \eqref{eq-quation-main-form} is the main equation we will study in this paper.
We note that $\bar{g}^{11}=z^2V=0$ on $z=0$. This causes a degeneracy in the operator $\Box_{\bar g}$ on $z=0$.

As an example, we consider the Schwarzschild metric $g_{S}$ given by
\begin{equation*}
g_{S}=-(1-\frac{2M}{r})dt^2+(1-\frac{2M}{r})^{-1}dr^2+r^{2}g_{S^2},
\end{equation*}
where $M$ is a constant.
Let $u$ satisfy the wave equation
\begin{equation*}
\Box_{g_{S}}u=0.
\end{equation*}
Set
\begin{equation*}
\tau=t-r-2M\log(r-2M), \quad z=\frac{1}{r}.
\end{equation*}
Then,
$$g_{S}=-(1-2Mz)d\tau^{2}+2z^{-2}d\tau dz+z^{-2}dg_{S^2},$$
and $v$ given by \eqref{eq-definition-v} satisfies
\begin{equation*}
2v_{z\tau}+z^{2}(1-2Mz)v_{zz}+\Delta_{S^2}v+2z(1-3Mz)v_{z}-2Mzv=0.
\end{equation*}

\section{$H^1$-Estimates}\label{sec-H1-Estimates}

In this section, we study a class of the timelike/null problem and derive $H^1$-estimates.

Consider the coordinates $\{x_{0}, x_{1}, x_{2}, x_{3}\}$, with $x_0=\tau$, $x_1=z$, and
$(x_2, x_3)$ forming local coordinates on $S^2$.
For some fixed $T>0$ and $z_0>0$, set
\begin{align*}
\Omega=\{ (\tau, z)|0<\tau< T, 0<z< z_{0}\}\times S^{2},
\end{align*}
and
\begin{align*}
\Sigma_{0}&=\{(0,z)|0<z\leq z_{0}\}\times S^{2},\\
\Sigma_{1}&=\{(\tau, z_{0})|0\leq\tau\leq T\}\times S^{2}.
\end{align*}
Let $V$, $U^A$, and $h_{AB}$ be functions on $\Omega$ such that, as $z\to 0$,
$$V\rightarrow 1, \quad U^A\rightarrow 0,$$
and
$$h_{AB}dx^Adx^B \rightarrow g_{S^2},$$
uniformly on $[0,T]\times S^2$, where $g_{S^2}$ is the standard round metric on $S^2$.

Let $g$ be the metric given by
\begin{equation}\label{eq-metric-g}
g=-z^2Vd\tau^{2}+2d\tau dz+h_{AB}(dx^{A}-U^{A}d\tau)(dx^{B}-U^{B}d\tau).
\end{equation}
The notations here and hereafter are different from those in the previous section:
$g$ and $h_{AB}$ in \eqref{eq-metric-g} are $\bar g$ and $e^{-2\eta}h_{AB}$ in \eqref{eq-metric-g-bar}.
If we write
$g=g_{ij}dx^idx^j$, then $g_{AB}=h_{AB}$,
\begin{align*}
&g_{00}=-z^{2}V+g_{AB}U^{A}U^{B}, \quad g_{0A}=-g_{AB}U^B,\\
&g_{01}=g_{10}=1, \quad g_{11}=g_{12}=g_{13}=0.
\end{align*}
Moreover, we assume
\begin{equation}\label{eq-assumption-V}\frac{1}{2}\leq V\leq\frac{3}{2}
\quad\text{in }\Omega\cup\Sigma_{0}\cup\Sigma_{1},\end{equation}
and, for some positive constants $\lambda$ and $\Lambda$,
\begin{equation}\label{eq-assumption-ellipticity}
\lambda I_2\le (h_{AB})\le \Lambda I_2\quad\text{in }\Omega.\end{equation}
We note that the Minkowski metric and the Schwarzschild metric satisfy these assumptions.
A straightforward computation yields
\begin{equation*}
g^{ij}=\begin{bmatrix}
0& 1&0 &0\\
1&z^{2}V&U^{2}&U^{3}\\
0&U^{2}&h^{22}&h^{23}\\
0&U^{3}&h^{32}&h^{33}
\end{bmatrix},
\end{equation*}
where $h_{AB}h^{BC}=\delta^C_A$.
Next, we set
\begin{equation}\label{eq-definition-gamma}
\gamma=\det(h_{AB}).\end{equation}
Then, $\gamma=-\det(g_{ij})$ and, by \eqref{eq-assumption-ellipticity},
$$\lambda\leq\sqrt{\gamma}\leq\Lambda\quad\text{in }\Omega.$$
We also have
\begin{align}\label{eq-expression-box}
\Box_gv=2\partial_{z\tau}v+z^{2}V\partial_{zz}v+2g^{1A}\partial_{zA}v+g^{AB}\partial_{AB}v
+\frac{1}{\sqrt{\gamma}}\partial_{i}(\sqrt{\gamma}g^{ij})\partial_{j}v.
\end{align}

Consider the operator
\begin{equation}\label{eq-operator-L}Lv=\Box_{g}v+a^{i}\partial_{i}v+\omega v,\end{equation}
for some given functions $a^i$ and $\omega$.
For given functions $f$ in $\Omega$, $\varphi$ on $\Sigma_0$ and $\psi$ on $\Sigma_1$,
we consider the following problem
\begin{align}\label{eq}\begin{split}
Lv&=f \quad\text{in } \Omega, \\
v|_{\Sigma_{0}}&=\varphi \quad \text{on } \Sigma_{0}, \\
v|_{\Sigma_{1}}&=\psi \quad \text{on } \Sigma_{1}.
\end{split}\end{align}

Next, we set
\begin{equation}\label{eq-definition-N}N_{1}=-\partial_{z},\quad
N_{2}=\partial_{\tau}+\frac{1}{2}z^{2}V\partial_{z}+g^{1A}\partial_{A}.\end{equation}
These two vector fields play an important role in this section.
It is easy to check
$$g(N_{1},N_{1})=0,\quad g(N_{2}, N_{2})=0,\quad g(N_{1}, N_{2})=-1.$$
Consider the hypersurfaces $\tau=const.$ and $z=const.$ Then,
\begin{equation}\label{eq-nabla-tau}\nabla \tau = \partial_z=-N_1,\end{equation}
and
\begin{equation}\label{eq-nabla-z}
\nabla z =\partial_{\tau}+z^{2}V\partial_{z}+g^{1A}\partial_{A}=-\frac12z^2VN_1+N_2.\end{equation}
Hence,
$$g(\nabla \tau, \nabla \tau)=0,\quad g(\nabla z, \nabla z)=z^{2}V.$$
Thus, the hypersurface $\tau=const.$ is null and the hypersurface $z=const.$ is timelike for $z>0$.

In \eqref{eq}, the initial value $\varphi$ is prescribe on the null hypersurface $\Sigma_0$ and
the boundary value $\psi$ is prescribed on the timelike hypersurface $\Sigma_1$.
Our primary goal in this section is to derive energy estimates.

\subsection{The Energy Momentum Tensor}
We first introduce an important quantity.
For a $C^1$-function $\phi$, the {\it associated energy momentum tensor} $Q[\phi]$ is a symmetric 2-tensor
defined by, for any vector fields $X$ and $Y$,
\begin{equation*}
Q[\phi](X,Y)=(X\phi)(Y\phi)-\frac{1}{2}g(X,Y)|\nabla\phi|^{2}-g(X,Y)\phi^{2}.
\end{equation*}

\begin{lemma}\label{lem}
Let $\phi$ be a given $C^{2}$-function and $Q[\phi]$ be the associated energy momentum tensor.
Set $X=a_{1}N_{1}+a_{2}N_{2}$ and $Y=a_{3}N_{1}+a_{4}N_{2}$, for some $a_i$, $i=1,2,3,4$. Then,
\begin{align*}
Q[\phi](X, Y)= a_{1}a_{3}(\partial_{z}\phi)^{2}+a_{2}a_{4}(N_{2}\phi)^{2}
+(a_{1}a_{4}+a_{2}a_{3})(\frac{1}{2}g^{AB}\partial_{A}\phi\partial_{B}\phi+\phi^{2}).
\end{align*}
Moreover, if $a_{i}\geq0$, $i=1,2,3,4$, then
$Q[\phi](X,Y)\geq0$.
\end{lemma}

\begin{proof}
First, we have
\begin{align*}
Q[\phi](X, Y)= a_{1}a_{3}Q[\phi](N_{1}, N_{1})+a_{2}a_{4}Q[\phi](N_{2}, N_{2})+(a_{1}a_{4}+a_{2}a_{3})Q[\phi](N_{1}, N_{2}).
\end{align*}
Next,
\begin{equation*}
|\nabla\phi|^{2}=2\partial_{\tau}\phi\partial_{z}\phi+z^{2}V(\partial_{z}\phi)^{2}+
2g^{1A}\partial_{z}\phi\partial_{A}\phi+g^{AB}\partial_{A}\phi\partial_{B}\phi,
\end{equation*}
and then
\begin{align*}
&Q[\phi](N_{1}, N_{1})=(\partial_{z}\phi)^{2}, \
Q[\phi](N_{2}, N_{2})=(N_{2}\phi)^{2}, \\
&Q[\phi](N_{1}, N_{2})=\frac{1}{2}g^{AB}\partial_{A}\phi\partial_{B}\phi+\phi^{2}.
\end{align*}
A simple substitution yields the desired results.
\end{proof}

Take a vector field
$$Y=Y^{0}\partial_{\tau}+Y^{1}\partial_{z}+Y^{A}\partial_{A}.$$
By \eqref{eq-nabla-tau}, \eqref{eq-nabla-z}, and Lemma \ref{lem}, we have
\begin{align}\label{ef1}\begin{split}
Q[\phi](\nabla\tau, Y)&= (Y^{1}-\frac{1}{2}z^{2}VY^{0})(\partial_{z}\phi)^{2}+
(Y^{A}-Y^{0}g^{1A})\partial_{z}\phi\partial_{A}\phi\\
&\qquad -\frac{1}{2}Y^{0}g^{AB}\partial_{A}\phi\partial_{B}\phi-Y^{0}\phi^{2},
\end{split}\end{align}
and
\begin{align}\label{ef2}\begin{split}
Q[\phi](\nabla z, Y)&= Y^{0}(\partial_{\tau}\phi)^{2}
+\frac{1}{2}z^{2}VY^{1}(\partial_{z}\phi)^{2}+z^{2}VY^{0}\partial_{\tau}\phi\partial_{z}\phi\\
&\qquad
+(Y^{A}+Y^{0}g^{1A})\partial_{A}\phi\partial_{\tau}\phi+z^{2}VY^{A}\partial_{A}\phi\partial_{z}\phi\\
&\qquad +(g^{1A}Y^{B}-\frac{1}{2}Y^{1})\partial_{A}\phi\partial_{B}\phi-Y^{1}\phi^{2}.
\end{split}\end{align}
By choosing $Y^0$ and $Y^1$ appropriately, we can make the coefficient of $(\partial_{z}\phi)^{2}$ in \eqref{ef1}
and the coefficient of $(\partial_{\tau}\phi)^{2}$ in \eqref{ef2} positive. It is natural to consider
a linear combination of $\tau$ and $z$.

We now introduce a quantity for integration by parts.
For a vector field $X$, the {\it deformation tensor} of $X$ is a symmetric 2-tensor $^{(X)}\pi$ defined by,
for any vector fields $Y$ and $Z$,
\begin{equation*}
^{(X)}\pi(Y,Z)=g(\nabla_{Y}X, Z)+g(\nabla_{Z}X, Y).
\end{equation*}
In local coordinates, we have
\begin{equation}\label{deftensor}
^{(X)}\pi^{\alpha\beta}=\partial^{\alpha}X^{\beta}+\partial^{\beta}X^{\alpha}-X(g^{\alpha\beta}).
\end{equation}

\begin{lemma}\label{lemma1}
Let $\phi$ be a $C^{2}$-function and $Q[\phi]$ be the associated energy momentum tensor.
Let $X$ be a vector field and set $P_{\alpha}=Q[\phi]_{\alpha\beta}X^{\beta}$. Then,
\begin{equation*}
\mathrm{div}P\equiv \nabla_{\alpha}P^{\alpha}
=(\Box_{g}\phi)(X\phi)+\frac{1}{2}\ Q[\phi]_{\alpha\beta}\ ^{(X)}\pi^{\alpha\beta}
-2g(X, \nabla\phi)\phi.
\end{equation*}
\end{lemma}

Refer to Page 33 \cite{Alinhac} for a proof.

For any $C^1$-function $h$, by Lemma \ref{lemma1}, we have
\begin{align}\label{divf}
\mathrm{div}(hP)
=(\Box_{g}\phi)(hX\phi)+\frac{1}{2}\ Q[\phi]_{\alpha\beta}\ ^{(hX)}\pi^{\alpha\beta}
-2g(X, \nabla\phi)h\phi,
\end{align}
where
\begin{equation*}
Q[\phi]_{\alpha\beta}\ ^{(hX)}\pi^{\alpha\beta}=hQ[\phi]_{\alpha\beta}\ ^{(X)}\pi^{\alpha\beta}+2Q[\phi](\nabla h, X).
\end{equation*}

We now start to derive energy estimates.
Balean \cite{Balean1997} used two different foliations of the spacetime
to obtain energy estimates. In this paper, we use weighted energy inequalities
so that interior terms can control all derivatives.

For functions defined on $\Sigma_0$, we can define the $H^p(\Sigma_0)$ space and the corresponding
$H^p(\Sigma_0)$-norm, with derivatives taken only with respect to variables on $\Sigma_0$. Similarly, we can
define the $H^p(\Sigma_1)$ space and the corresponding
$H^p(\Sigma_1)$-norm.

We prove an $H^1$-estimate of solutions in this section
and derive higher order estimates in the next section. There are two issues related to the special features of \eqref{eq}.
First, we do not assume solutions are regular up to $z=0$
and hence we cannot integrate in the entire domain $\Omega$ directly.
We need to choose a special domain away from $z=0$ by a small
distance and then take limit. Second, estimates of domain integrals should be in terms of
initial values on $\Sigma_0$ and boundary values on $\Sigma_1$. For example,
with $\varphi$ given as a $C^1$-function on $\Sigma_0$, $\partial_z \varphi$ and $\nabla_{S^2}v$
are considered to be known functions. However, $\varphi$ does not
yield any information on $\partial_\tau v|_{\Sigma_0}$.

\begin{theorem}\label{maintheo} Suppose $f\in L^{2}(\Omega)$,
$\varphi\in H^{1}(\Sigma_{0})$, and $\psi\in H^{1}(\Sigma_{1})$. Let $v$ be a
$C^2(\Omega)\cap C^1(\Omega\cup\Sigma_0\cup\Sigma_1)$-solution of
\eqref{eq}.
Then,
\begin{equation}\label{eq-estimate-H1}
\|v\|_{H^{1}(\Omega)}+\|\partial_{z}v\|_{L^{2}(\Sigma_1)}\leq C\{\|f\|_{L^{2}(\Omega)}
+\|\varphi\|_{H^{1}(\Sigma_{0})}+\|\psi\|_{H^{1}(\Sigma_{1})}\},
\end{equation}
where $C$ is a positive constant depending only on $z_0$, $T$, $\lambda$, $\Lambda$,
$|g^{ij}|_{C^{1}(\Omega)}$, $ |a^{\alpha}|_{L^{\infty}(\Omega)}$, and
$|\omega|_{L^{\infty}(\Omega)}$.
\end{theorem}

\begin{proof} The proof consists of several steps.

{\it Step 1. Derivation of an integral identity.}
For a constant $\epsilon\in (0, z_0)$,
we consider
\begin{equation*}
\Omega_{\epsilon}=\{(\tau, z)| 0<\tau< T, z(\tau)< z< z_{0}\}\times S^{2},
\end{equation*}
where, $z(\tau)$, for $\tau\in[0,T]$, is a function to be chosen later, with the property $z(0)=\epsilon$.

For some constants $m, p, q>0$ to be chosen later, we consider a timelike vector field
\begin{equation}\label{eq-definition-Y}Y= N_{2}-m\partial_{z}=mN_1+N_2,\end{equation} and functions
\begin{equation}\label{eq-definition-h}w=-p\tau+qz,\quad h= e^{w}.\end{equation}
Set
$$P_{\alpha}=Q[v](\partial_{\alpha}, Y),\quad  P^{\alpha}=g^{\alpha\beta}P_{\beta}.$$
For the vector field $P$, an integration over $\Omega_{\epsilon}$ yields
\begin{equation}\label{sf}
\int_{\Omega_{\epsilon}}\mathrm{div}(hP)\  d\Omega=\int_{\partial\Omega_{\epsilon}}i_{hP}\ d(\partial\Omega),
\end{equation}
where $d\Omega$ and $d(\partial\Omega)$ are the volume element and the area element, respectively. Note
$d\Omega=\sqrt{\gamma}dzd\tau d\Sigma$, $\gamma$ is introduced in \eqref{eq-definition-gamma}, and
$d\Sigma$ is the area element on $S^2$.
A straightforward calculation yields
\begin{align*}
i_{hP}d\Omega&=-\sqrt{\gamma}hP^{0}dzd\Sigma=-\sqrt{\gamma}hQ[v](\nabla\tau, Y)dzd\Sigma
\quad\text{on }\tau=0\text{ or }T,\\
i_{hP}d\Omega&=\sqrt{\gamma}h[Q[v](\nabla z, Y)-z'(\tau)Q[v](\nabla\tau, Y)]d\tau d\Sigma
\quad\text{on }z=z(\tau),\\
i_{hP}d\Omega&=\sqrt{\gamma}hP^{1}dzd\Sigma=\sqrt{\gamma}hQ[v](\nabla z, Y)d\tau d\Sigma
\quad\text{on }z=z_0.
\end{align*}
By (\ref{divf}), \eqref{eq-operator-L}, and (\ref{sf}), we obtain
\begin{align}\label{keyf}\begin{split}
&\int_{\Omega_{\epsilon}}\{Q[v](\nabla h, Y)
+\frac{1}{2}hQ[v]_{\alpha\beta}\ ^{(Y)}\pi^{\alpha\beta}-2g(Y, \nabla v)hv\\
&\quad\qquad\qquad +(Lv-a^{\alpha}\partial_{\alpha}v-\omega v)h(Yv)\}d\Omega\\
&\qquad= \int_{\tau=T, z\geq z(T)}hQ[v](\nabla\tau, Y)\sqrt{\gamma}dzd\Sigma\\
&\qquad\quad -\int_{z=z(\tau)}h[Q[v](\nabla z, Y)-z'(\tau)Q[v](\nabla\tau, Y)]\sqrt{\gamma} d\tau d\Sigma\\
&\qquad\quad -\int_{\tau=0, z\geq\epsilon}hQ[v](\nabla\tau, Y)\sqrt{\gamma}dzd\Sigma
+\int_{z=z_{0}}hQ[v](\nabla z, Y)\sqrt{\gamma}d\tau d\Sigma.
\end{split}\end{align}
This is the key formula.

\smallskip
{\it Step 2. Analysis of domain integrals.} We now analyze the expression in the left-hand side of \eqref{keyf}.
We consider $p=lq$ for $p$ and $q$ in \eqref{eq-definition-h}, for some constant $l>0$.
We claim that,  for $l>0$ sufficiently large depending on
$m$, $\lambda$, $\Lambda$, and $|g^{1A}|_{L^\infty(\Omega_{T})}$,
\begin{equation}\label{eq-estimate-Q}
Q[v](\nabla w, Y)\geq \frac{1}{2}q\{(\partial_{\tau}v)^{2}+(\partial_{z}v)^{2}+g^{AB}\partial_{A}v\partial_{B}v+v^{2}\}.
\end{equation}
To prove \eqref{eq-estimate-Q}, we first note, by \eqref{eq-nabla-tau} and \eqref{eq-nabla-z},
$$\nabla w=-p\nabla \tau+q\nabla z=(p-\frac12qz^2V)N_1+qN_2.$$
By \eqref{eq-definition-N}, \eqref{eq-definition-Y} and Lemma \ref{lem}, we have
\begin{align}\label{1}\begin{split}
Q[v](\nabla w, Y)&= Q[v]([p-\frac{1}{2}z^{2}Vq]N_1+qN_{2}, mN_1+N_{2})\\
&= q(\partial_{\tau}v)^{2}+[mp+\frac{1}{2}qz^{2}V(\frac{1}{2}z^{2}V-m)](\partial_{z}v)^{2}\\
&\qquad +(mq+p-\frac{1}{2}qz^{2}V)[\frac12g^{AB}\partial_{A}v\partial_{B}v+v^{2}]\\
&\qquad +qz^{2}V\partial_{\tau}v\partial_{z}v+2qg^{1A}\partial_{A}v\partial_{\tau}v\\
&\qquad+qz^{2}Vg^{1A}\partial_{A}v\partial_{z}v+q(g^{1A}\partial_{A}v)^{2}.
\end{split}\end{align}
By the Cauchy inequality, we have
\begin{align}\label{2}\begin{split}
qz^{2}V\partial_{\tau}v\partial_{z}v&\leq q\{\frac{1}{4}(\partial_{\tau}v)^{2}+C(\partial_{z}v)^{2}\},\\
2qg^{1A}\partial_{A}v\partial_{\tau}v&\leq q\{\frac{1}{4}(\partial_{\tau}v)^{2}
+Cg^{AB}\partial_{A}v\partial_{B}v\},\\
qz^{2}Vg^{1A}\partial_{A}v\partial_{z}v&\leq Cq\{(\partial_{z}v)^{2}
+Cg^{AB}\partial_{A}v\partial_{B}v\},
\end{split}\end{align}
where $C$ is a positive constant depending on $\lambda$, $\Lambda$, and $|g^{1A}|_{L^\infty(\Omega_{T})}$.
By (\ref{1}) and (\ref{2}), we get
\begin{align*}
Q[v](\nabla w, Y)&\geq \frac{1}{2}q(\partial_{\tau}v)^{2}
+(mp-Cmq-Cq)(\partial_{z}v)^{2}\\
&\qquad +[\frac{1}{2}(mq+p)-Cq]g^{AB}\partial_{A}v\partial_{B}v
+(mq+p-Cq)v^{2}\\
&\geq  \frac{1}{2}q\{(\partial_{\tau}v)^{2}+(\partial_{z}v)^{2}+g^{AB}\partial_{A}v\partial_{B}v+v^{2}\},
\end{align*}
by choosing $p=lq$, for $l$ sufficiently large.
This finishes the proof \eqref{eq-estimate-Q}.

By (\ref{deftensor}) and the definition of $Q$, we have
\begin{equation*}
Q[v]_{\alpha\beta}\ ^{(Y)}\pi^{\alpha\beta}
=\sum\limits_{i,j=0}^{3}a_{ij}\partial_{i}v\partial_{j}v+\sum\limits_{i=0}^{3}b_{i}v\partial_{i}v+cv,
\end{equation*}
where $a_{ij}, b_{i}, c$ depend on $g$ and $\partial g$. By \eqref{eq-estimate-Q}
and choosing $q$
sufficiently large, depending on $|g^{ij}|_{C^{1}(\Omega_{T})}$,
$|a^{\alpha}|_{L^{\infty}(\Omega_{T})}$, and $|\omega|_{L^{\infty}(\Omega_{T})}$, we obtain
\begin{align}\label{eq-estimate-domain-integral}\begin{split}
& \int_{\Omega_{T,\epsilon}}\big\{Q[v](\nabla h, Y)
+\frac{1}{2}hQ[v]_{\alpha\beta}\ ^{(Y)}\pi^{\alpha\beta}-2g(Y, \nabla v)hv\\
 &\qquad\quad+(Lv-a^{\alpha}\nabla_{\alpha}v-\omega v)h(Yv)\big\}d\Omega\\
&\quad \geq \frac{1}{4}\int_{\Omega_{T,\epsilon}}h\{q[(\partial_{\tau}v)^{2}
+(\partial_{z}v)^{2}+g^{AB}\partial_{A}v\partial_{B}v+v^{2}]-(Lv)^{2}\}d\Omega.
\end{split}\end{align}

{\it Step 3. Analysis of boundary integrals.} We now analyze the expression in the right-hand side of \eqref{keyf}.
By (\ref{ef1}), (\ref{ef2}) and
\eqref{eq-definition-Y}, we have
\begin{equation}\label{3}
Q[v](\nabla\tau, Y)= -m(\partial_{z}v)^{2}-\frac{1}{2}g^{AB}\partial_{A}v\partial_{B}v-v^{2},
\end{equation}
and
\begin{equation}\label{3a}\begin{split}
Q[v](\nabla z, Y)&= (N_2v)^{2}+\frac{1}{2}mg^{AB}\partial_{A}v\partial_{B}v+mv^{2}\\
&\qquad -\frac12z^2V[m(\partial_{z}v)^{2}+\frac{1}{2}g^{AB}\partial_{A}v\partial_{B}v+v^{2}].
\end{split}\end{equation}
Note
\begin{equation}\label{eq-Q-T}Q[v](\nabla\tau, Y)\leq 0\quad\text{on }\tau=T.\end{equation}
Next, we consider those terms on $z=z(\tau)$. By \eqref{3} and \eqref{3a}, we have
\begin{align*}
&Q[v](\nabla z, Y)-z'(\tau)Q[v](\nabla\tau, Y)\\
&\qquad= (N_2v)^{2}+\frac{1}{2}mg^{AB}\partial_{A}v\partial_{B}v+mv^{2}\\
&\qquad\quad +(z'(\tau)-\frac12z^2(\tau)V)[m(\partial_{z}v)^{2}+\frac{1}{2}g^{AB}\partial_{A}v\partial_{B}v+v^{2}].
\end{align*}
We choose
\begin{equation}\label{eq-definition-z}
z(\tau)=\frac{\epsilon}{1-M\tau\epsilon}\quad\text{for } \tau\in[0, T],\end{equation}
where $M=3/4$. Then, by \eqref{eq-assumption-V},
$$z'(\tau)\geq \frac{1}{2}z^{2}(\tau)V.$$
Hence,
\begin{equation}\label{eq-estimate-z=z(tau)}Q[v](\nabla z, Y)-z'(\tau)Q[v](\nabla\tau, Y)\geq 0\quad\text{on }z=z(\tau).
\end{equation}
For terms on $z=z_{0}$, by expanding $(N_2v)^2$ in (\ref{3a}), we get
\begin{align*}
Q[v](\nabla z, Y)&=(\partial_{\tau}v+g^{1A}\partial_{A}v)^{2}
+ (m-\frac12z^2V)(\frac{1}{2}g^{AB}\partial_{A}v\partial_{B}v+v^{2})\\
&\qquad-\frac12z^2V(m-\frac{1}{2}z^{2}V)(\partial_{z}v)^{2}
+z^{2}V\partial_{\tau}v\partial_{z}v
+z^{2}Vg^{1A}\partial_{A}v\partial_{z}v.
\end{align*}
By the Cauchy inequality and
choosing $m$ large enough depending on $|g^{1A}|_{L^\infty(\Omega_{T})}$, we get
\begin{equation}\label{4}
Q[v](\nabla z, Y)\leq C\{(\partial_{\tau}v)^{2}+g^{AB}\partial_{A}v\partial_{B}v+v^{2}\}-(\partial_{z}v)^{2},
\end{equation}
where $C$ is a positive constant depending on $\lambda$, $\Lambda$, and $|g^{1A}|_{L^\infty(\Omega_{T})}$.

\smallskip
{\it Step 4. Completion of the proof.}
Take any $\epsilon>0$. By (\ref{keyf}), \eqref{eq-estimate-domain-integral},
\eqref{3}, \eqref{eq-Q-T}, and \eqref{eq-estimate-z=z(tau)},
we have
\begin{align}\label{eq-estimate-q}\begin{split}
&\frac{1}{4}\int_{\Omega_{\epsilon}}h\{q[(\partial_{\tau}v)^{2}+(\partial_{z}v)^{2}
+g^{AB}(\partial_{A}v)(\partial_{B}v)+v^{2}]-f^{2}\}d\Omega\\
&\qquad\leq -\int_{\tau=0, z\geq\epsilon}hQ(\nabla\tau, Y)\sqrt{\gamma}dzd\Sigma
+\int_{z=z_{0}}hQ(\nabla z, Y)\sqrt{\gamma}d\tau d\Sigma.
\end{split}\end{align}
Hence, by (\ref{3}), (\ref{4}), and $e^{-pT}\leq h\leq e^{qz_0}$, we obtain
\begin{align*}
& \int_{\Omega_{\epsilon}}\left[(\partial_{\tau}v)^{2}+(\partial_{z}v)^{2}
+g^{AB}\partial_{A}v\partial_{B}v+v^{2}\right]d\Omega+\int_{\Sigma_1}(\partial_{z}v)^{2}d\tau d\Sigma\\
&\qquad\leq  C\left\{\int_{\Omega_{\epsilon}}f^{2}d\Omega+\|\varphi\|_{H^{1}(\Sigma_{0})}^{2}
+\|\psi\|_{H^{1}(\Sigma_{1})}^{2}\right\},
\end{align*}
where $C$ does not depend on $\epsilon$. Taking $\epsilon\rightarrow0$, we finish the proof.
\end{proof}

\begin{corollary}\label{cor-radiation-L2}
Under the assumptions of Theorem \ref{maintheo}, $v_0=\lim_{z\to 0} v$ exists as an
$H^1$-function on $[0,T]\times S^2$ and
\begin{equation}\label{eq-estimate-radiation-L2}
\|v_0\|_{H^1([0,T]\times S^2)}\leq C\{\|f\|_{L^{2}(\Omega)}
+\|\varphi\|_{H^{1}(\Sigma_{0})}+\|\psi\|_{H^{1}(\Sigma_{1})}\}.
\end{equation}
\end{corollary}

\begin{proof} In Step 3 above, we actually have, on $z=z(\tau)$,
\begin{align*}
Q[v](\nabla z, Y)-z'(\tau)Q[v](\nabla\tau, Y)\ge (N_2v)^{2}+\frac{1}{2}mg^{AB}\partial_{A}v\partial_{B}v+mv^{2}.
\end{align*}
Note that $N_2v\to \partial_\tau v$ as $z\to 0$.
\end{proof}

We point out that $v_0$ is the radiation field. (Refer to \cite{Friedlander1967}.)

We have another version of the $H^1$-estimate.

\begin{theorem}\label{maintheo-another} Suppose $f\in L^{2}(\Omega)$,
$\varphi\in H^{1}(\Sigma_{0})$, and $\psi\in H^{1}(\Sigma_{1})$. Let $v$ be a
$C^2(\Omega)\cap C^1(\Omega\cup\Sigma_0\cup\Sigma_1)$-solution of
\eqref{eq}.
Then,
\begin{equation}\label{eq-estimate-H1-alternative}
\|v\|_{H^{1}(\Omega)}+\|\partial_{z}v\|_{L^{2}(\Sigma_1)}\leq C\{\|f\|_{L^{2}(\Omega)}
+\|\varphi\|_{H^{1}(\Sigma_{0})}+\|N_{2}v\|_{L^{2}(\Sigma_{1})}\},
\end{equation}
where $C$ is a positive constant depending only on $z_{0}$, $T$, $\lambda$, $\Lambda$,
$|g^{ij}|_{C^{1}(\Omega)}$, $ |a^{\alpha}|_{L^{\infty}(\Omega)}$, and
$|\omega|_{L^{\infty}(\Omega)}$.
\end{theorem}

\begin{proof}
We only need to analyze $Q[v](\nabla z, Y)$ on $z=z_0$ in Step 3 differently.
By (\ref{3a}), we rewrite $Q[v](\nabla z, Y)$ as
\begin{equation*}
Q[v](\nabla z, Y)=(N_{2}v)^{2}+\frac{1}{2}(m-\frac{1}{2}z^{2}V)[g^{AB}(\partial_{A}v)(\partial_{B}v)+2v^{2}]
-\frac{1}{2}z^{2}Vm(\partial_{z}v)^{2}.
\end{equation*}
Taking $m=\frac{1}{4}z_{0}^{2}$, we have
\begin{equation}\label{a8}
Q[v](\nabla z, Y)\leq (N_{2}v)^{2}-\frac{1}{2}z^{2}Vm(\partial_{z}v)^{2}.
\end{equation}
This implies \eqref{eq-estimate-H1-alternative}. We note that we take $m$ small here instead of
large in the proof of Theorem \ref{maintheo}.
\end{proof}

Theorem \ref{maintheo} and Theorem \ref{maintheo-another} hold under a weaker assumption
$v\in H^2(\Omega\cap \{z>\varepsilon\})$, for any $\varepsilon>0$.
It is important {\it not} to assume that solutions are continuous up to $z=0$
or have $L^2$-traces on $z=0$. An important application of these theorems
and the corresponding higher order estimates is to allow the passage from analytic solutions to
sufficiently smooth solutions.
Analytic solutions will be proved to exist only for $z>0$ and can be extended up to $z=0$ {\it after} they satisfy
the energy estimates. See also Remark \ref{remark-Regularity}.

We now rewrite \eqref{eq-estimate-H1} and
\eqref{eq-estimate-H1-alternative}. Let  $h$ the positive function defined in  \eqref{eq-definition-h}.
Define
$$\|u\|_{H^p_h(\Omega)}=\sum_{i=0}^p\|h\nabla^iu\|_{L^2(\Omega)},$$
and similarly $\|u\|_{H^p_h(\Sigma_0)}$ and $\|u\|_{H^p_h(\Sigma_1)}$.
Here, we simply insert a factor of $h^2$ in calculating the $L^2$-norms.
In the proof of Theorem \ref{maintheo}, by multiplying the equation by $h^2$ instead of $h$ and
renaming $q$, we obtain
\begin{equation}\label{eq-estimate-H1-h}
q\|v\|_{H^{1}_h(\Omega)}+\|\partial_{z}v\|_{L^{2}_h(\Sigma_1)}\leq C\{\|f\|_{L^{2}_h(\Omega)}
+\|\varphi\|_{H^{1}_h(\Sigma_{0})}+\|\psi\|_{H^{1}_h(\Sigma_{1})}\},
\end{equation}
where $C$ is a constant independent of $q$. Similarly, we have
\begin{equation}\label{eq-estimate-H1-alternative-h}
q\|v\|_{H^{1}_h(\Omega)}+\|\partial_{z}v\|_{L^{2}_h(\Sigma_1)}\leq C\{\|f\|_{L^{2}_h(\Omega)}
+\|\varphi\|_{H^{1}_h(\Sigma_{0})}+\|N_{2}v\|_{L^{2}_h(\Sigma_{1})}\}.
\end{equation}
In the next section, when we apply \eqref{eq-estimate-H1-h} and \eqref{eq-estimate-H1-alternative-h},
we have freedom to choose $q$ large.

\section{Higher Order Estimates}\label{sec-Hk-Estimates}

In this section, we continue to study the timelike/null problem \eqref{eq}
and derive $H^k$-estimates, for $k\ge 2$.
To this end, we need to differentiate the equation \eqref{eq}.
The following commutation formula is helpful.

\begin{lemma}\label{commulemma}
Let $X$ be a field with a deformation tensor $\pi=^{(X)}\pi$. Then,
\begin{equation*}
[\Box,X]\phi=\pi^{\alpha\beta}\nabla^{2}\phi_{\alpha\beta}
+\nabla_{\alpha}\pi^{\alpha\beta}\partial_{\beta}\phi
-\frac{1}{2}\partial^{\alpha}(tr\pi)\partial_{\alpha}\phi.
\end{equation*}
\end{lemma}

Refer to P59-61  \cite{Alinhac} for a proof.

\smallskip

In the following, we denote the multi-indices $\alpha, \beta\in \mathbb Z^4_+$
by $\alpha=(\alpha_0, \alpha_1, \alpha_2, \alpha_3)$,
$\beta=(\beta_0, \beta_1, \beta_2, \beta_3)$, etc.

For a fixed multi-index $\alpha$, by applying appropriate vector fields to the equation
(\ref{eq}) successively and using
Lemma \ref{commulemma}, we can derive an equation for $\partial^\alpha v$ given by
\begin{align}\label{eq-equation-general}
\Box_{g}(\partial^\alpha v)=f_\alpha,
\end{align}
where $f_\alpha$ is a linear combination of $\partial^\alpha f$ and certain
derivatives of  $v$ up to order $|\alpha|+1$. A precise form will be given later.
We attempt to apply the derived $H^1$-estimates
to this equation and get an estimate of the $H^1$-norm of $\partial^\alpha v$.
There are two issues we need to resolve. First, $f_\alpha$ contains derivatives of $v$ of order $|\alpha|+1$,
not all of which can be written as $\partial_i\partial^\alpha v$, for some $i=0, 1, 2, 3$. If we
simply apply the derived $H^1$-estimates, there are derivatives of order $|\alpha|+1$ in the right-hand side,
which are not yet controlled. Second, we need to determine
the initial values of $\partial^\alpha v$ on $\Sigma_0$ and
the boundary values on $\Sigma_1$.

To resolve the first issue, we fix an integer $p$ and look for a set $I$ of multi-indices $\alpha$ with $|\alpha|=p$ such that,
for each $\alpha\in I$, derivatives of order $p+1$ in the right-hand side of the equation \eqref{eq-equation-general}
can be written as $\partial_i\partial^\beta v$, for some $\beta\in I$ and  $i=0, 1, 2, 3$.
In a sense, the system of equations \eqref{eq-equation-general} for $\alpha\in I$ forms a {\it closed} system.
For the Minkowski metric, the system for derivatives with respect to only spherical coordinates forms a closed system,
and the system for derivatives with respect to spherical coordinates and $z$ also forms a closed system.
In this case, we can estimate derivatives with respect to only spherical coordinates first,
then  derivatives with respect to spherical coordinates and $z$, and last arbitrary derivatives. Refer to \cite{Balean1997}.
For the general metric, the system for derivatives with respect to only spherical coordinates
does not necessarily form a closed system. However, the system for derivatives with respect
to spherical coordinates and $z$ forms a closed system. We will start with this system.

Now, we turn to the second issue concerning initial values and boundary values.
We can always estimate  $\|\partial^\alpha v\|_{H^1(\Sigma_0)}$ in terms of known quantities.
However, such an estimate exhibits a loss of differentiation if $\partial^\alpha v$
involves $\tau$-derivatives.
Boundary values exhibit different features.
We cannot always estimate
$\|\partial^\alpha v\|_{H^1(\Sigma_1)}$ or
$\|N_2(\partial^\alpha v)\|_{L^2(\Sigma_1)}$ in terms of known quantities.
It turns out that, for a closed system for a set $I$ of multi-indices and each $\alpha\in I$, we can estimate
$\|\partial^\alpha v\|_{H^1(\Sigma_1)}$ or
$\|N_2(\partial^\alpha v)\|_{L^2(\Sigma_1)}$ in terms of some quantities we intend to estimate.
In other words, quantities we intend to control in the left-hand side will appear also in the right-hand side.
We need to make a balance and control these quantities.

We first estimate derivatives of $v$ with respect to $\tau$ at most once.

\begin{theorem}\label{theorem-partial-estia}
For some integer $k\ge 2$, suppose $f\in H^{k-1}(\Omega)$,
$\varphi\in H^{k}(\Sigma_{0})$, and $\psi\in H^{k}(\Sigma_{1})$. Let $v$ be a
$C^{k+1}(\Omega)\cap C^{k}(\Omega\cup\Sigma_0\cup\Sigma_1)$-solution of
\eqref{eq}.
Then, for any $|\alpha|\le k$ with $\alpha_0\le 1$,
\begin{align*}
\|\partial^{\alpha}v\|_{L^2(\Omega)}
\leq C\{\|f\|_{H^{k-1}(\Omega)}+\|\varphi\|_{H^{k}(\Sigma_{0})}+\|\psi\|_{H^{k}(\Sigma_{1})}\},
\end{align*}
where
$C$ is a positive constant depending only on $k$, $T$, $\lambda$, $\Lambda$,
$|g^{ij}|_{C^{k}(\Omega)}$, $ |a^{\alpha}|_{C^{k-1}(\Omega)}$, and
$|\omega|_{C^{k-1}(\Omega)}$.
\end{theorem}

\begin{proof}
For any $p=1, \cdots, k-1$, we  prove
\begin{align}\label{eq-estimate-iteration-0}\begin{split}
&\sum_{\substack{|\beta|=p+1\\ \beta_0\le 1}}\|\partial^\beta v\|_{L^2(\Omega)}
+\sum_{\substack{|\beta|=p+1\\ \beta_0\le 1}}\|\partial^\beta v\|_{L^2(\Sigma_1)}\\
&\quad\leq  C\big\{\sum_{\substack{|\beta|\le p\\ \beta_0\le 1}}\|\partial^\beta v\|_{L^2(\Omega)}
+\sum_{\substack{|\beta|\le p\\ \beta_0\le 1}}\|\partial^\beta v\|_{L^2(\Sigma_1)}\\
&\quad\qquad+\|f\|_{H^{p}(\Omega)}+\|\varphi\|_{H^{p+1}(\Sigma_0)}
+\|\psi\|_{H^{p+1}(\Sigma_{1})}\big\}.
\end{split}\end{align}
Then, the desired result follows from a simple iteration and Theorem \ref{maintheo}.
The proof of (\ref{eq-estimate-iteration-0}) consists of several steps.

{\it Step 1. Equations of derivatives of $v$.}
Take a multi-index $\alpha=(\alpha_0, \alpha_1, \alpha_2, \alpha_3)$ with $|\alpha|=p$ and $\alpha_0=0$.
By \eqref{eq-expression-box}, \eqref{eq-operator-L} and Lemma \ref{commulemma},
we have
\begin{align}\label{eq-equation-beta-0}
\Box_{g}(\partial^\alpha v)=f_\alpha,
\end{align}
where
\begin{align*}
f_{\alpha}= \partial^\alpha f
+\sum_{|\beta|\le p+1}c_{\alpha, \beta}\partial^\beta v.
\end{align*}
In \eqref{eq-equation-beta-0}, all derivatives of $v$ of order $p+2$ are in $\Box_{g}(\partial^\alpha v)$,
and hence the highest order of derivatives of $v$ in $f_\alpha$ is $p+1$.
In each term involving $v$ in $f_\alpha$, the derivative with respect to $\tau$ is at most 1. Hence, we can write
\begin{align}\label{eq-expression-f-beta-0}
f_{\alpha}=\partial^\alpha f
+\sum_{\substack{|\beta|\le p+1 \\ \beta_0\le 1}}c_{\alpha,\beta}\partial^\beta v.
\end{align}

{\it Step 2. Estimates of $\partial^\beta v$, with $|\beta|=p+1$ and $\beta_0\le 1$.}
First, consider a multi-index $\alpha$ with $|\alpha|=p$ and $\alpha_0=\alpha_1=0$.
Applying \eqref{eq-estimate-H1-h} to \eqref{eq-equation-beta-0}, we have
\begin{equation*}
q\|\partial^\alpha v\|_{H^{1}_h(\Omega)}+\|\partial_{z}\partial^\alpha v\|_{L^{2}_h(\Sigma_1)}
\leq C\{\|f_\alpha\|_{L^{2}_h(\Omega)}
+\|\partial^\alpha v\|_{H^{1}_h(\Sigma_{0})}+\|\partial^\alpha v\|_{H^{1}_h(\Sigma_{1})}\}.
\end{equation*}
For the domain integral in the left-hand side, we keep only derivatives of order $p+1$.
For the right-hand side, we note that
$\partial^\beta v=\partial^\beta \varphi$ on $\Sigma_0$ if $\beta_0=0$ and
$\partial^\beta v=\partial^\beta \psi$ on $\Sigma_1$ if $\beta_1=0$.
Hence,
\begin{align}\label{eq-estimate-p-0-0}
q\sum_{\substack{|\alpha|= p \\ \alpha_0=\alpha_1=0}}\|\partial\partial^\alpha v\|_{L^2_h(\Omega)}
+\sum_{\substack{|\alpha|= p \\ \alpha_0=\alpha_1= 0}}\|\partial_z\partial^\alpha v\|_{L^{2}_h(\Sigma_1)}\le
CM,\end{align}
where
\begin{align}\label{eq-definition-M-0}\begin{split}
M&=\sum_{\substack{|\beta|\le p+1 \\ \beta_0\le 1}}\|\partial^\beta v\|_{L^2_h(\Omega)}
+\sum_{\substack{|\beta|\le p \\ \beta_0\le 1}}\|\partial^\beta v\|_{L^2_h(\Sigma_1)}\\
&\qquad+\|f\|_{H^{p}_h(\Omega)}
+\|\varphi\|_{H^{p+1}_h(\Sigma_{0})}+\|\psi\|_{H^{p+1}_h(\Sigma_{1})}.\end{split}\end{align}
We point out that not all derivatives of order $p+1$ in $M$ are included in the left-hand side of \eqref{eq-estimate-p-0-0}.
The summation of the boundary integrals are inserted for later purpose.

Next, take $1\le m\le p$ and consider a multi-index $\alpha$ with $|\alpha|=p$, $\alpha_0=0$ and $\alpha_1=m$.
Applying \eqref{eq-estimate-H1-alternative-h} to \eqref{eq-equation-beta-0}, we have
\begin{equation*}
q\|\partial^\alpha v\|_{H^{1}_h(\Omega)}+\|\partial_{z}\partial^\alpha v\|_{L^{2}_h(\Sigma_1)}
\leq C\{\|f_\alpha\|_{L^{2}_h(\Omega)}
+\|\partial^\alpha v\|_{H^{1}_h(\Sigma_{0})}+\|N_2(\partial^\alpha v)\|_{L^2_h(\Sigma_{1})}\}.
\end{equation*}
Then similarly,
\begin{align}\label{eq-estimate-p-m-0}\begin{split}
&q\sum_{\substack{|\alpha|= p \\ \alpha_0=0,\alpha_1=m}}\|\partial\partial^\alpha v\|_{L^2_h(\Omega)}
+\sum_{\substack{|\alpha|= p \\ \alpha_0=0,\alpha_1= m}}\|\partial_z\partial^\alpha v\|_{L^{2}_h(\Sigma_1)}\\
&\qquad\le C\big\{M
+\sum_{\substack{|\alpha|= p \\ \alpha_0=0,\alpha_1=m}}\|N_2(\partial^\alpha v)\|_{L^2_h(\Sigma_{1})}\big\}.
\end{split}\end{align}

{\it Step 3. Boundary integrals.} By keeping only the boundary integrals in \eqref{eq-estimate-p-0-0}
and \eqref{eq-estimate-p-m-0}, we have
\begin{align}\label{eq-estimate-p-0-0-boundary}
\sum_{\substack{|\beta|= p+1 \\ \beta_0=0, \beta_1= 1}}\|\partial^\beta v\|_{L^{2}_h(\Sigma_1)}\le CM,\end{align}
and
\begin{align}\label{eq-estimate-p-m-0-boundary}
\sum_{\substack{|\beta|= p+1 \\ \beta_0=0,\beta_1= m+1}}\|\partial^\beta v\|_{L^{2}_h(\Sigma_1)}
\le C\big\{M
+\sum_{\substack{|\alpha|= p \\ \alpha_0=0,\alpha_1=m}}\|N_2(\partial^\alpha v)\|_{L^2_h(\Sigma_{1})}\big\}.
\end{align}
We now analyze $N_2(\partial^\alpha v)$
in \eqref{eq-estimate-p-m-0-boundary}. Take $\alpha'$ with $|\alpha'|=p-1$ such that
$\partial^\alpha v=\partial^{\alpha'}\partial_zv$.
Then, with $\alpha$ replaced by $\alpha'$ in \eqref{eq-equation-beta-0}, we have
\begin{align}\label{eq-equation-beta-p-1-0}
\Box_{g}(\partial^{\alpha'} v)=\partial^{\alpha'} f
+\sum_{\substack{|\beta|\le p \\ \beta_0\le 1}}c_{\alpha,\beta}\partial^\beta v.
\end{align}
Note
$$\Box_{g}(\partial^{\alpha'} v)=N_2(\partial^{\alpha'} \partial_zv)+g^{AB}\partial_{AB}\partial^{\alpha'} v+\cdots,$$
where the omitted terms involve derivatives of order at most
$p$ as in the right-hand side of \eqref{eq-equation-beta-p-1-0}.
Hence,
\begin{align*}
N_2(\partial^{\alpha} v)=\partial^{\alpha'} f
+\sum_{\substack{|\beta|=p+1 \\ \beta_0=0, \beta_1=m-1}}c_{\alpha,\beta}\partial^\beta v
+\sum_{\substack{|\beta|\le p \\ \beta_0\le 1}}c_{\alpha,\beta}\partial^\beta v.
\end{align*}
We now restrict this identity to $\Sigma_1$ and take the $L^2$-norm. By the trace theorem, we have
$$\|\partial^{\alpha'} f\|_{L^2_h(\Sigma_{1})}\le C\|f\|_{H^p_h(\Omega)}.$$
This is a part of $M$. The boundary $L^2$-integral of $\partial^\beta v$, with $|\beta|\le p$ and $\beta_0\le 1$,
is also a part of $M$.
Therefore, for $1\le m\le p$,
\begin{align}\label{eq-equation-N-beta-p-1-0}
\sum_{\substack{|\alpha|= p \\ \alpha_0=0,\alpha_1=m}}
\|N_2(\partial^{\alpha} v)\|_{L^2_h(\Sigma_{1})}
\le C\big\{M
+\sum_{\substack{|\beta|=p+1 \\ \beta_0=0, \beta_1=m-1}}\|\partial^\beta v\|_{L^2_h(\Sigma_{1})}\big\}.
\end{align}
Similarly,
for $1\le m\le p$,
\begin{align}\label{eq-equation-N-tau-p-1-0}
\sum_{\substack{|\alpha|= p \\ \alpha_0=0,\alpha_1=m}}
\|\partial_\tau\partial^{\alpha} v\|_{L^2_h(\Sigma_{1})}
\le C\big\{M
+\sum_{\substack{|\beta|=p+1 \\ \beta_0=0, \beta_1\le m+1}}\|\partial^\beta v\|_{L^2_h(\Sigma_{1})}\big\}.
\end{align}

If $m=1$, then the second term in the right-hand side
in \eqref{eq-equation-N-beta-p-1-0}, with $\beta_1=0$, is bounded by
$\|\psi\|_{H^{p+1}_h(\Sigma_{1})}$, which is
a part of $M$. Hence,
\begin{align}\label{eq-equation-N-boundary-p-1-0}
\sum_{\substack{|\alpha|= p \\ \alpha_0=0,\alpha_1=1}}
\|N_2(\partial^{\alpha} v)\|_{L^2_h(\Sigma_{1})}\le CM.
\end{align}
With $m=1$ in \eqref{eq-estimate-p-m-0-boundary}, we have
\begin{align*}
\sum_{\substack{|\beta|= p+1 \\ \beta_0=0,\beta_1= 2}}\|\partial^\beta v\|_{L^{2}_h(\Sigma_1)}
\le CM.
\end{align*}
With $m=2$ in \eqref{eq-equation-N-beta-p-1-0}, we have
\begin{align}\label{eq-equation-N-boundary-p-2-0}
\sum_{\substack{|\alpha|= p \\ \alpha_0=0,\alpha_1=2}}
\|N_2(\partial^{\alpha} v)\|_{L^2_h(\Sigma_{1})}\le CM.
\end{align}
We note that $\alpha_1$ in \eqref{eq-equation-N-boundary-p-2-0} improves by 1, compared with
$\alpha_1$ in \eqref{eq-equation-N-boundary-p-1-0}. By an iteration with \eqref{eq-equation-N-beta-p-1-0}
and \eqref{eq-estimate-p-m-0-boundary}, we obtain, for any $1\le m\le p$.
\begin{align}\label{eq-equation-N-boundary-p-m-0-general}
\sum_{\substack{|\alpha|= p \\ \alpha_0=0,\alpha_1=m}}
\|N_2(\partial^{\alpha} v)\|_{L^2_h(\Sigma_{1})}\le CM.
\end{align}
By substituting \eqref{eq-equation-N-boundary-p-m-0-general} in \eqref{eq-estimate-p-m-0}, we have, for any $1\le m\le p$.
\begin{align}\label{eq-estimate-p-m-0-general}
q\sum_{\substack{|\alpha|= p \\ \alpha_0=0,\alpha_1=m}}\|\partial\partial^\alpha v\|_{L^2_h(\Omega)}
+\sum_{\substack{|\alpha|= p \\ \alpha_0=0,\alpha_1= m}}\|\partial_z\partial^\alpha v\|_{L^{2}_h(\Sigma_1)}
\le CM.
\end{align}

{\it Step 4. Completion of the proof.}
By adding \eqref{eq-estimate-p-0-0} and \eqref{eq-estimate-p-m-0-general} for $m=1, \cdots, p$,
we obtain
\begin{align*}
q\sum_{\substack{|\beta|= p+1 \\ \beta_0\le 1}}\|\partial^\beta v\|_{L^2_h(\Omega)}
+\sum_{\substack{|\beta|= p+1 \\ \beta_0=0}}\|\partial^\beta v\|_{L^{2}_h(\Sigma_1)}
\le CM.
\end{align*}
Combining with \eqref{eq-equation-N-tau-p-1-0}, we have
\begin{align}\label{eq-estimate-p-m-0-general-sum}
q\sum_{\substack{|\beta|= p+1 \\ \beta_0\le 1}}\|\partial^\beta v\|_{L^2_h(\Omega)}
+\sum_{\substack{|\beta|= p+1 \\ \beta_0\le 1}}\|\partial^\beta v\|_{L^{2}_h(\Sigma_1)}
\le CM.
\end{align}
Now, the domain integrals of derivatives of $v$ of order $p+1$ in $M$ is the same as those in the left-hand side.
By choosing $q$ sufficiently large, we can absorb those terms in $M$. Then, we fix such a $q$ and remove $h$
from all integrals. By the definition of $M$ in \eqref{eq-definition-M-0}, we obtain
\begin{align}\label{eq-estimate-iteration-0-sum}\begin{split}
&\sum_{\substack{|\beta|=p+1\\ \beta_0\le 1}}\|\partial^\beta v\|_{L^2(\Omega)}
+\sum_{\substack{|\beta|=p+1\\ \beta_0\le 1}}\|\partial^\beta v\|_{L^2(\Sigma_1)}\\
&\quad\leq  C\big\{\sum_{\substack{|\beta|\le p\\ \beta_0\le 1}}\|\partial^\beta v\|_{L^2(\Omega)}
+\sum_{\substack{|\beta|\le p\\ \beta_0\le 1}}\|\partial^\beta v\|_{L^2(\Sigma_1)}\\
&\quad\qquad+\|f\|_{H^{p}(\Omega)}+\|\varphi\|_{H^{p+1}(\Sigma_0)}
+\|\psi\|_{H^{p+1}(\Sigma_{1})}\big\}.
\end{split}\end{align}
This is \eqref{eq-estimate-iteration-0}. \end{proof}

We point out that different versions of $H^1$-estimates
are applied to equations of $\partial^\alpha v$, with $|\alpha|=p$.
Specifically, we employ \eqref{eq-estimate-H1-h}
for $\alpha_1=0$
and employ \eqref{eq-estimate-H1-alternative-h} for $\alpha_1\ge 1$.
For $\alpha_1\ge 1$, it is easier to estimate
$\|N_2(\partial^{\alpha}v)\|_{L^2(\Sigma_1)}$ than $\|\partial^{\alpha}v\|_{H^1(\Sigma_1)}$.

\begin{corollary}\label{cor-radiation-Hk-j1}
Under the assumptions of Theorem \ref{theorem-partial-estia}, $v_i=\lim_{z\to 0} \partial _z^iv$ exists as an
$H^1$-function on $[0,T]\times S^2$, for any $i=0, \cdots, k-1$, and, for any
$j, m$ with $i+j+m\le k$ and $j\le 1$,
\begin{equation}\label{eq-estimate-radiation-Hk-j1}
\|\partial_\tau^j \partial_{S^2}^m v_i\|_{L^2([0,T]\times S^2)}
\leq C\{\|f\|_{H^{k-1}(\Omega)}+\|\varphi\|_{H^{k}(\Sigma_{0})}+\|\psi\|_{H^{k}(\Sigma_{1})}\}.
\end{equation}
\end{corollary}

The existence of $v_0$ is already addressed in Corollary \ref{cor-radiation-L2}.

\smallskip

Now we estimate derivatives of $v$ of higher order with respect to $\tau$.

\begin{theorem}\label{theorem-Hp-estimates-0}
For some integers $k\ge l\ge 2$, suppose $f\in H^{k+l-2}(\Omega)$,
$\varphi\in H^{k+l-1}(\Sigma_{0})$, and $\psi\in H^{k+l-1}(\Sigma_{1})$. Let $v$ be a
$C^{k+l}(\Omega)\cap C^{k+l-1}(\Omega\cup\Sigma_0\cup\Sigma_1)$-solution of
\eqref{eq}.
Then, for any $|\alpha|\le k$ with $\alpha_0\le l$,
\begin{align*}
\|\partial^{\alpha}v\|_{L^2(\Omega)}
\leq C\{\|f\|_{H^{k+l-2}(\Omega)}+\|\varphi\|_{H^{k+l-1}(\Sigma_{0})}+\|\psi\|_{H^{k+l-1}(\Sigma_{1})}\},
\end{align*}
where
$C$ is a positive constant depending only on $k$, $T$, $\lambda$, $\Lambda$,
$|g^{ij}|_{C^{k+l-1}(\Omega)}$, $ |a^{\alpha}|_{C^{k+l-2}(\Omega)}$, and
$|\omega|_{C^{k+l-2}(\Omega)}$.
\end{theorem}

\begin{proof}
We will prove, for $1\le l {\color{red} <} p\le k-1$,
\begin{align}\label{eq-estimate-iteration-l}\begin{split}
&\sum_{\substack{|\beta|=p+1\\ \beta_0\le l+1}}\|\partial^\beta v\|_{L^2(\Omega)}
+\sum_{\substack{|\beta|=p+1\\ \beta_0\le l+1}}\|\partial^\beta v\|_{L^2(\Sigma_1)}\\
&\quad\leq  C\big\{\sum_{\substack{|\beta|\le p\\ \beta_0\le l+1}}\|\partial^\beta v\|_{L^2(\Omega)}
+\sum_{\substack{|\beta|\le p\\ \beta_0\le l+1}}\|\partial^\beta v\|_{L^2(\Sigma_1)}\\
&\quad\qquad+\|f\|_{H^{p+l}(\Omega)}+\|\varphi\|_{H^{p+l+1}(\Sigma_0)}
+\|\psi\|_{H^{p+l+1}(\Sigma_{1})}\big\},
\end{split}\end{align}
and, for $1\le p=l\le k-1$,
\begin{align}\label{eq-estimate-iteration-l=p}\begin{split}
&\|\partial^{p+1}_\tau v\|_{L^2(\Omega)}
+\|\partial_z\partial^p_\tau v\|_{L^2(\Sigma_1)}\\
&\quad\leq  C\big\{\sum_{\substack{|\beta|\le p+1\\ \beta_0\le p}}\|\partial^\beta v\|_{L^2(\Omega)}
+\|f\|_{H^{2p}(\Omega)}+\|\varphi\|_{H^{2p+1}(\Sigma_0)}
+\|\psi\|_{H^{2p+1}(\Sigma_{1})}\big\}.
\end{split}\end{align}
Then, the desired result follows from a simple iteration and Theorem \ref{theorem-partial-estia}.
The proof of (\ref{eq-estimate-iteration-l}) and \eqref{eq-estimate-iteration-l=p}
consists of several steps.

{\it Step 1. Equations of derivatives of $v$ and initial values on $\Sigma_0$.}
Take an arbitrary multi-index $\alpha$ with $|\alpha|=p$ and $\alpha_0=l$. Then,
\begin{align}\label{eq-equation-beta}
\Box_{g}(\partial^\alpha v)=f_\alpha,
\end{align}
where
\begin{align}\label{eq-expression-f-beta}
f_\alpha=\partial^\alpha f
+\sum_{\substack{|\beta|\le p+1\\ \beta_0\le l+1}}c_\beta \partial^\beta v.
\end{align}
In particular, if $p=l$, we have
\begin{align}\label{eq-expression-f-beta-p=l}
f_\alpha=\partial^\alpha f+a\partial_\tau^{p+1}v
+\sum_{\substack{|\beta|\le p+1\\ \beta_0\le p}}c_\beta \partial^\beta v.
\end{align}
In fact, there is only $\alpha$ with $|\alpha|=\alpha_0=p$; namely, $\alpha=(p,0,0,0)$.

Next, by restricting the equation $Lu=f$ to $\Sigma_0$, we have
\begin{equation}\label{eq-ODE-tau}
2\partial_{z}(\partial_{\tau}v)+a^{0}\partial_\tau v=f_1,
\end{equation}
where
$$f_1=f+\sum_{\substack{|\beta|\le 2\\ \beta_0=0}}a_\beta\partial^{\beta}v.$$
We point out that no derivatives of $v$ with respect to $\tau$ appear in $f_1$.
We view \eqref{eq-ODE-tau} as an ODE of $\partial_\tau v$ in $z$ on $\Sigma_0$ with the initial value given by
$$\partial_\tau v=\partial_\tau \psi\quad\text{on }\Sigma_0\cap \Sigma_1.$$
Then,
\begin{align*}
\partial_\tau v=\partial_\tau \psi e^{-\frac{1}{2}\int_{z}^{z_{0}}a^{0}dz'}
-\frac12\int_z^{z_0}f_1e^{\frac{1}{2}\int_{z'}^{z_{0}}a^{0}dz''}dz'
\quad\text{on }\Sigma_0.\end{align*}
Therefore,
\begin{equation}\label{initdata-1z}
\|\partial_\tau v\|_{L^{2}(\Sigma_{0})}
\leq C\left\{\|\varphi\|_{H^{2}(\Sigma_{0})}
+\|\partial_\tau \psi\|_{L^2(\Sigma_{0}\cap \Sigma_{1})}+\|f\|_{L^2(\Sigma_{0})}\right\}.
\end{equation}
For $l\ge 2$, by applying $\partial_{\tau}^{l-1}$ to \eqref{eq-ODE-tau}, we obtain
\begin{equation*}
2\partial_{z}(\partial_{\tau}^{l}v)+a\partial_{\tau}^{l}v=f_{l},
\end{equation*}
where
$$f_{l}=\partial_{\tau}^{l-1}f
+\sum_{\substack{|\beta|\le l+1\\ \beta_0\le l-1}}c_{\beta}\partial^{\beta}v.$$
Similarly, we view this as an ODE of $\partial_{\tau}^{l}v$ in $z$ on $\Sigma_0$
with the initial value given  by
$$\partial_{\tau}^{l}v=\partial_{\tau}^{l}\psi\quad\text{on }\Sigma_0\cap \Sigma_1.$$
Then,
\begin{align*}
\partial^{l}_\tau v=\partial^{l}_\tau \psi e^{-\frac{1}{2}\int_{z}^{z_{0}}adz'}
-\frac12\int_z^{z_0}f_{l}e^{\frac{1}{2}\int_{z'}^{z_{0}}adz''}dz'
\quad\text{on }\Sigma_0.\end{align*}
For
$\alpha=(l, \alpha_1, \alpha_2, \alpha_3)$, we write $\alpha'=(0, \alpha_1, \alpha_2, \alpha_3)$. Then,
\begin{align*}
\partial^\alpha v=
\partial^{\alpha'}\big\{\partial^{l}_\tau \psi e^{-\frac{1}{2}\int_{z}^{z_{0}}adz'}
-\frac12\int_z^{z_0}f_{l}e^{-\frac12\int_{z'}^{z_{0}}adz''}dz'\big\}
\quad\text{on }\Sigma_0.\end{align*}
Hence, for $\alpha$ with $|\alpha|=p$ and $\alpha_0=l$,
\begin{align*}
\|\partial^{\alpha}v\|_{L^{2}(\Sigma_{0})}
\leq C\big\{\|f\|_{H^{p-1}(\Sigma_{0})}
+\|\partial_\tau^{l}\psi\|_{H^{p-l}(\Sigma_{0}\cap \Sigma_{1})}
+\sum_{\substack{|\beta|\le p+1\\ \beta_0\le l-1}}\|\partial^{\beta}v\|_{L^{2}(\Sigma_{0})}
\big\}.
\end{align*}
By the trace theorem,
we have
\begin{align*}
\|\partial^{\alpha}v\|_{L^{2}(\Sigma_{0})}
\leq C\big\{\|f\|_{H^{p}(\Omega)}
+\|\psi\|_{H^{p+1}(\Sigma_{1})}
+\sum_{\substack{|\beta|\le p+1\\ \beta_0\le l-1}}\|\partial^{\beta}v\|_{L^{2}(\Sigma_{0})}
\big\}.
\end{align*}
We note that in the summation above, the highest degree of derivatives increases by 1 but the highest
degree of derivatives with respect to $\tau$ decreases by 1.
 So we can iterate this inequality $l$ times and obtain, for $\alpha$ with $|\alpha|=p$ and $\alpha_0=l$,
\begin{align}\label{initdata}
\|\partial^{\alpha}v\|_{L^{2}(\Sigma_{0})}
\leq C\big\{\|f\|_{H^{p+l-1}(\Omega)}+\|\varphi\|_{H^{p+l}(\Sigma_{0})}
+\|\psi\|_{H^{p+l}(\Sigma_{1})}\big\},
\end{align}
where $C$ is a positive constant depending only on $|\beta|$, $|g^{ij}|_{C^{p+l-1}(\Sigma_{0})}$,
$|a^{i}|_{C^{p+l-2}(\Sigma_{0})}$, and $|\omega|_{C^{p+l-2}(\Sigma_{0})}$.

The rest of the proof is similar as the proof of Theorem \ref{theorem-partial-estia}.
\smallskip

{\it Step 2. Estimates of $\partial^\beta v$, with $|\beta|=p+1$ and $\beta_0=l,l+1$.}
First, we consider a multi-index $\alpha$ with $|\alpha|=p$, $\alpha_0=l$ and $\alpha_1=0$.
By applying (\ref{eq-estimate-H1-h}) to (\ref{eq-equation-beta}), we have
\begin{align*}
q\|\partial^\alpha v\|_{H^{1}_h(\Omega)}+\|\partial_{z}\partial^\alpha v\|_{L^{2}_h(\Sigma_1)}
\leq C\{\|f_\alpha\|_{L^{2}_h(\Omega)}
+\|\partial^\alpha v\|_{H^{1}_h(\Sigma_{0})}+\|\partial^\alpha v\|_{H^{1}_h(\Sigma_{1})}\}.
\end{align*}
We  note that $\partial^\beta v=\partial^\beta \psi$ on $\Sigma_1$ if $\beta_1=0$.  By (\ref{initdata}), we get
\begin{align}\label{eq-estimate-p-l-0-general}\begin{split}
&q\|\partial\partial^\alpha v\|_{L^2_h(\Omega)}
+\|\partial_{z}\partial^\alpha v\|_{L^{2}_h(\Sigma_1)}\\
&\qquad\leq C\{\|f_\alpha\|_{L^{2}_h(\Omega)}
+\|f\|_{H^{p+l}_h(\Omega)}
+\|\varphi\|_{H^{p+l+1}_h(\Sigma_{0})}+\|\psi\|_{H^{p+l+1}_h(\Sigma_{1})}\big\}.
\end{split}
\end{align}

We first consider the case $l=p$. Then, $\alpha=(p,0,0,0)$. We  keep only the $L^2$-norm of
$\partial^{p+1}_\tau v$ in the left-hand side and use it to absorb the same term in $f_\alpha$ given by
\eqref{eq-expression-f-beta-p=l} by choosing $q$ large.
Then, we fix such a $q$ and remove $h$
from all integrals. By \eqref{eq-expression-f-beta-p=l}, we have
\begin{align*}
&\|\partial^{p+1}_\tau v\|_{L^2(\Omega)}
+\|\partial_z\partial^p_\tau v\|_{L^2(\Sigma_1)}\\
&\quad\leq  C\big\{\sum_{\substack{|\beta|\le p+1\\ \beta_0\le p}}\|\partial^\beta v\|_{L^2(\Omega)}
+\|f\|_{H^{2p}(\Omega)}+\|\varphi\|_{H^{2p+1}(\Sigma_0)}
+\|\psi\|_{H^{2p+1}(\Sigma_{1})}\big\}.
\end{align*}
This is \eqref{eq-estimate-iteration-l=p}.

In the following, we consider $l<p$. By \eqref{eq-estimate-p-l-0-general} and a simple summation, we have
\begin{align}\label{eq-estimate-p-l-0}
q\sum_{\substack{|\alpha|= p \\ \alpha_0=l, \alpha_1=0}}\|\partial\partial^\alpha v\|_{L^2_h(\Omega)}
+\sum_{\substack{|\alpha|= p \\ \alpha_0=l,  \alpha_1= 0}}\|\partial_z\partial^\alpha v\|_{L^{2}_h(\Sigma_1)}\le
CM,\end{align}
where
\begin{align}\label{eq-definition-M-l-0}\begin{split}
M&=\sum_{\substack{|\beta|\le p+1 \\ \beta_0\le l+1}}\|\partial^\beta v\|_{L^2_h(\Omega)}
+\sum_{\substack{|\beta|\le p \\ \beta_0\le l+1}}\|\partial^\beta v\|_{L^2_h(\Sigma_1)}\\
&\qquad+\|f\|_{H^{p+l}_h(\Omega)}
+\|\varphi\|_{H^{p+l+1}_h(\Sigma_{0})}+\|\psi\|_{H^{p+l+1}_h(\Sigma_{1})}.\end{split}\end{align}

Next, take $1\le m\le p$ and consider a multi-index $\alpha$ with $|\alpha|=p$, $\alpha_0=l$ and
$\alpha_1=m$. Applying (\ref{eq-estimate-H1-alternative-h}) to (\ref{eq-equation-beta}), we have
\begin{equation*}
q\|\partial^\alpha v\|_{H^{1}_h(\Omega)}+\|\partial_{z}\partial^\alpha v\|_{L^{2}_h(\Sigma_1)}
\leq C\{\|f_\alpha\|_{L^{2}_h(\Omega)}
+\|\partial^\alpha v\|_{H^{1}_h(\Sigma_{0})}+\|N_2(\partial^\alpha v)\|_{L^2_h(\Sigma_{1})}\}.
\end{equation*}
Then similarly,
\begin{align}\label{eq-estimate-p-m-l-0}\begin{split}
&q\sum_{\substack{|\alpha|= p \\ \alpha_0=l,\alpha_1=m}}\|\partial\partial^\alpha v\|_{L^2_h(\Omega)}
+\sum_{\substack{|\alpha|= p \\ \alpha_0=l,\alpha_1= m}}\|\partial_z\partial^\alpha v\|_{L^{2}_h(\Sigma_1)}\\
&\qquad\le C\big\{M
+\sum_{\substack{|\alpha|= p \\ \alpha_0=l,\alpha_1=m}}\|N_2(\partial^\alpha v)\|_{L^2_h(\Sigma_{1})}\big\}.
\end{split}\end{align}

{\it Step 3. Boundary integrals.}
By keeping only the boundary integrals in \eqref{eq-estimate-p-l-0}
and \eqref{eq-estimate-p-m-l-0}, we have
\begin{align}\label{eq-estimate-p-0-l-boundary}
\sum_{\substack{|\beta|= p+1 \\ \beta_0=l, \beta_1= 1}}\|\partial^\beta v\|_{L^{2}_h(\Sigma_1)}\le CM,\end{align}
and
\begin{align}\label{eq-estimate-p-m-l-boundary}
\sum_{\substack{|\beta|= p+1 \\ \beta_0=l,\beta_1= m+1}}\|\partial^\beta v\|_{L^{2}_h(\Sigma_1)}
\le C\big\{M
+\sum_{\substack{|\alpha|= p \\ \alpha_0=l,\alpha_1=m}}\|N_2(\partial^\alpha v)\|_{L^2_h(\Sigma_{1})}\big\}.
\end{align}
Similar to \eqref{eq-equation-N-beta-p-1-0} and \eqref{eq-equation-N-tau-p-1-0}, we have for $1\le m\le p$,
\begin{align}\label{eq-equation-N-beta-p-1-0-l}
\sum_{\substack{|\alpha|= p \\ \alpha_0=l,\alpha_1=m}}
\|N_2(\partial^{\alpha} v)\|_{L^2_h(\Sigma_{1})}
\le C\big\{M
+\sum_{\substack{|\beta|=p+1 \\ \beta_0=l, \beta_1=m-1}}\|\partial^\beta v\|_{L^2_h(\Sigma_{1})}\big\},
\end{align}
and
\begin{align}\label{eq-equation-N-tau-p-1-0-l}
\sum_{\substack{|\alpha|= p \\ \alpha_0=l,\alpha_1=m}}
\|\partial_\tau\partial^{\alpha} v\|_{L^2_h(\Sigma_{1})}
\le C\big\{M
+\sum_{\substack{|\beta|=p+1 \\ \beta_0=l, \beta_1\le m+1}}\|\partial^\beta v\|_{L^2_h(\Sigma_{1})}\big\}.
\end{align}

If $m=1$, then the second term in \eqref{eq-equation-N-beta-p-1-0-l}, with $\beta_1=0$, is bounded by
$\|\psi\|_{H^{p+1}_h(\Sigma_{1})}$, which is
a part of $M$. Hence,
\begin{align}\label{eq-equation-N-boundary-p-1-0-l}
\sum_{\substack{|\alpha|= p \\ \alpha_0=l,\alpha_1=1}}
\|N_2(\partial^{\alpha} v)\|_{L^2_h(\Sigma_{1})}\le CM.
\end{align}
With $m=1$ in \eqref{eq-estimate-p-m-l-boundary}, we have
\begin{align*}
\sum_{\substack{|\beta|= p+1 \\ \beta_0=l,\beta_1= 2}}\|\partial^\beta v\|_{L^{2}_h(\Sigma_1)}
\le CM.
\end{align*}
With $m=2$ in \eqref{eq-equation-N-beta-p-1-0-l}, we have
\begin{align}\label{eq-equation-N-boundary-p-2-0-l}
\sum_{\substack{|\alpha|= p \\ \alpha_0=l,\alpha_1=2}}
\|N_2(\partial^{\alpha} v)\|_{L^2_h(\Sigma_{1})}\le CM.
\end{align}
We note that $\alpha_1$ in \eqref{eq-equation-N-boundary-p-2-0-l} improves by 1, compared with
$\alpha_1$ in \eqref{eq-equation-N-boundary-p-1-0-l}. By an iteration with \eqref{eq-equation-N-beta-p-1-0-l}
and \eqref{eq-estimate-p-m-l-boundary}, we obtain, for any $1\le m\le p$.
\begin{align}\label{eq-equation-N-boundary-p-m-l-general}
\sum_{\substack{|\alpha|= p \\ \alpha_0=l,\alpha_1=m}}
\|N_2(\partial^{\alpha} v)\|_{L^2_h(\Sigma_{1})}\le CM.
\end{align}
By substituting \eqref{eq-equation-N-boundary-p-m-l-general} in \eqref{eq-estimate-p-m-l-0}, we have, for any $1\le m\le p$.
\begin{align}\label{eq-estimate-p-m-l-general}
q\sum_{\substack{|\alpha|= p \\ \alpha_0=l,\alpha_1=m}}\|\partial\partial^\alpha v\|_{L^2_h(\Omega)}
+\sum_{\substack{|\alpha|= p \\ \alpha_0=l,\alpha_1= m}}\|\partial_z\partial^\alpha v\|_{L^{2}_h(\Sigma_1)}
\le CM.
\end{align}

{\it Step 4. Completion of the proof.}
By adding (\ref{eq-estimate-p-l-0}) and (\ref{eq-estimate-p-m-l-general}) for $m=1, \cdot\cdot\cdot,p$, we have
\begin{align*}
q\sum_{\substack{|\beta|= p+1 \\ \beta_0=l,l+1}}\|\partial^\beta v\|_{L^2_h(\Omega)}
+\sum_{\substack{|\beta|= p+1 \\ \beta_0=l}}\|\partial^\beta v\|_{L^{2}_h(\Sigma_1)}
\le CM.
\end{align*}
Combining with \eqref{eq-equation-N-tau-p-1-0-l}, we have
\begin{align*}
q\sum_{\substack{|\beta|= p+1 \\ \beta_0=l,l+1}}\|\partial^\beta v\|_{L^2_h(\Omega)}
+\sum_{\substack{|\beta|= p+1 \\ \beta_0=l, l+1}}\|\partial^\beta v\|_{L^{2}_h(\Sigma_1)}
\le CM.
\end{align*}
Note that the above estimate obviously holds for $\beta_0\le l-1$. Hence,
\begin{align*}
q\sum_{\substack{|\beta|= p+1 \\ \beta_0\le l+1}}\|\partial^\beta v\|_{L^2_h(\Omega)}
+\sum_{\substack{|\beta|= p+1 \\ \beta_0\le l+1}}\|\partial^\beta v\|_{L^{2}_h(\Sigma_1)}
\le CM.
\end{align*}
Now, the domain integrals of derivatives of $v$ of order $p+1$ in $M$ is the same as those in the left-hand side.
By choosing $q$ sufficiently large, we can absorb those terms in $M$. Then, we fix such a $q$ and remove $h$
from all integrals. By the definition of $M$ in \eqref{eq-definition-M-l-0}, we obtain
\begin{align*}
&\sum_{\substack{|\beta|=p+1\\ \beta_0\le l+1}}\|\partial^\beta v\|_{L^2(\Omega)}
+\sum_{\substack{|\beta|=p+1\\ \beta_0\le l+1}}\|\partial^\beta v\|_{L^2(\Sigma_1)}\\
&\quad\leq  C\big\{\sum_{\substack{|\beta|\le p\\ \beta_0\le l+1}}\|\partial^\beta v\|_{L^2(\Omega)}
+\sum_{\substack{|\beta|\le p\\ \beta_0\le l+1}}\|\partial^\beta v\|_{L^2(\Sigma_1)}\\
&\quad\qquad+\|f\|_{H^{p+l}(\Omega)}+\|\varphi\|_{H^{p+l+1}(\Sigma_0)}
+\|\psi\|_{H^{p+l+1}(\Sigma_{1})}\big\}.
\end{align*}
This is \eqref{eq-estimate-iteration-l}.
\end{proof}

\begin{corollary}\label{cor-radiation-Hk-jl}
Under the assumptions of Theorem \ref{theorem-Hp-estimates-0}, for
$v_i$ as in Corollary \ref{cor-radiation-Hk-j1}, for $i=0, \cdots, k-1$, and, for any
$j, m$ with $i+j+m\le k$ and $j\le l$,
\begin{equation}\label{eq-estimate-radiation-Hk-jl}
\|\partial_\tau^j \partial_{S^2}^m v_i\|_{L^2([0,T]\times S^2)}
\leq C\{\|f\|_{H^{k+l-2}(\Omega)}+\|\varphi\|_{H^{k+l-1}(\Sigma_{0})}+\|\psi\|_{H^{k+l-1}(\Sigma_{1})}\}.
\end{equation}
\end{corollary}

In general, we have the following result.

\begin{corollary}\label{theorem-Hp-estimates}
For some integers $k\ge 2$, suppose $f\in H^{2k-2}(\Omega)$,
$\varphi\in H^{2k-1}(\Sigma_{0})$, and $\psi\in H^{2k-1}(\Sigma_{1})$. Let $v$ be a
$C^{2k}(\Omega)\cap C^{2k-1}(\Omega\cup\Sigma_0\cup\Sigma_1)$-solution of
\eqref{eq}.
Then,
\begin{align*}
\|v\|_{H^k(\Omega)}
\leq C\{\|f\|_{H^{2k-2}(\Omega)}+\|\varphi\|_{H^{2k-1}(\Sigma_{0})}+\|\psi\|_{H^{2k-1}(\Sigma_{1})}\},
\end{align*}
where
$C$ is a positive constant depending only on $k$, $T$, $\lambda$, $\Lambda$,
$|g^{ij}|_{C^{2k-1}(\Omega)}$, $ |a^{\alpha}|_{C^{2k-2}(\Omega)}$, and
$|\omega|_{C^{2k-2}(\Omega)}$.  Moreover, $v_i\in H^{k-i}([0,T]\times S^2)$ for $i=0, \cdots, k-1$, and
\begin{align}\label{eq-estimate-radiation-Hk}
\|v_i\|_{H^{k-i}([0,T]\times S^2)}
\leq C\{\|f\|_{H^{2k-2}(\Omega)}+\|\varphi\|_{H^{2k-1}(\Sigma_{0})}+\|\psi\|_{H^{2k-1}(\Sigma_{1})}\}.
\end{align}
\end{corollary}

We point out that there is a loss of regularity by {\it half} along the $\tau$-direction.
This is a well-known fact for characteristic initial-value problems.

\section{Existence of Solutions and Their Radiation Fields}\label{sec-Existence}

We now consider the existence of solutions of the timelike-null problem \eqref{eq}.
We will adapt a well-known process of proving the existence for the analytic case following the method in \cite{Duff1958},
and then obtaining the existence for the general case by
approximations. See also \cite{Hagen1977}.

We adopt the setting in the previous section and consider
\begin{align}\label{eq-main-eq}\begin{split}
Lv&=0 \quad \text{in } \Omega, \\
v|_{\Sigma_{0}}&=\varphi \quad \text{on } \Sigma_{0}, \\
v|_{\Sigma_{1}}&=\psi \quad \text{on } \Sigma_{1}.
\end{split}\end{align}

\begin{theorem}\label{theorem-Existence}
For some integer $m\ge 2$,  assume $\varphi\in H^{2m+5}(\Sigma_0)$ and
$\psi\in H^{2m+5}(\Sigma_{1})$, with $\varphi=\psi$ on $\Sigma_0\cap \Sigma_1$.
Then,
there exists a unique solution $v\in C^{m}(\Omega_{T}\cup \Sigma_0\cup \Sigma_1)$ of \eqref{eq-main-eq}.
Moreover,
$$\|v\|_{C^m(\Omega)}\le C\big\{\|\varphi\|_{H^{2m+5}(\Sigma_0)}
+\|\psi\|_{H^{2m+5}(\Sigma_1)}\big\},$$
where $C$ is a positive constant depending on $m$, $\lambda$, $\Lambda$, $T$, $z_0$, and  $g^{ij}$.
\end{theorem}

\begin{proof}
We write the equation in \eqref{eq-main-eq} in the form
\begin{equation}
2\partial_{z\tau}v+z^{2}V\partial_{zz}v+2g^{1A}\partial_{zA}v+g^{AB}\partial_{AB}v+\tilde{a}^{i}\partial_{i}v+bv=0.
\end{equation}

We first assume that $g^{ij}, \tilde{a}^{i}, b,$ and $\varphi,\psi$ are real analytic and hence can be expanded
as a power series in $\tau$. For example, we have
\begin{align*}
\psi=\sum_{i=0}^{\infty}\psi_{i}( \theta)\tau^{i}.
\end{align*}
Here and hereafter, we denote by $\theta=(x^2, x^3)$, coordinates on $S^2$. We define
\begin{align*}
&u^{0}=v=\sum_{i=0}^{\infty} u_{i}^{0}(z, \theta)\tau^{i},\\
&u^{1}=\partial_{z}v=\sum_{i=0}^{\infty} u_{i}^{1}(z, \theta)\tau^{i},\\
&u^{A}=\partial_{A}v=\sum_{i=0}^{\infty} u_{i}^{A}(z, \theta)\tau^{i}, \ A=2,3, \end{align*}
and
\begin{align*}
w=\partial_{\tau}v=\sum_{i=0}^{\infty} w_{i}(z, \theta)\tau^{i}.
\end{align*}
For $u^0, u^1, u^2, u^3$, and $ w$, we have
\begin{align*}
\partial_\tau u^{0}&=w,\\
2\partial_\tau u^{1}&=-z^{2}V\partial_z u^{1}-2g^{1A}\partial_z u^{A}
-g^{AB}\partial_B u^{A}+\bar{a}^{i}u^{i}+cw,\\
\partial_\tau u^{A}&=\partial_A w,\end{align*}
and
\begin{align*}
2\partial_z w=-z^{2}V\partial_z u^{1}-2g^{1A}\partial_z u^{A}
-g^{AB}\partial_B u^{A}+\bar{a}^{i}u^{i}+cw.
\end{align*}
Therefore,
\begin{align}\label{eq-ODE-system-2}\begin{split}
(i+1)u_{i+1}^{0}&=w_{i},\\
2(i+1)u_{i+1}^{1}&=\sum_{k\leq i}
L_{k}[\partial u_{k}^{0}, \partial u_{k}^{1}, \partial u_{k}^{2}, \partial u_{k}^{3}, w_{k}]
,\\
(i+1)u_{i+1}^{A}&=\partial_A w_{i},
\end{split}\end{align}
and
\begin{align}\label{eq-ODE-system-1}
2\partial_z w_{i}+d w_{i}
=\sum_{k\leq i}L_{k}[\partial u_{k}^{0}, \partial u_{k}^{1}, \partial u_{k}^{2}, \partial u_{k}^{3}]
+\sum_{k\leq i-1}L_{k}[w_{k}].\end{align}
Note
$$u_{0}^{0}=\varphi, \quad u_{0}^{1}=\partial_z \varphi, \quad u_{0}^{A}=\partial_A \varphi.$$
The equation \eqref{eq-ODE-system-1} is an ODE of $w_{i}$ with respect to $z$
and the initial value is given by
$$w_{i}(z_0, \cdot)=(i+1)\psi_{i+1}\quad\text{on }S^2.$$
For some $i\ge 0$, assume we already know $u_0^l$, $\cdots$, $u_i^l$, for $l=0, 1, 2, 3$,
and $w_0$, $\cdots$, $w_{i-1}$, then we can find $w_i$ by solving \eqref{eq-ODE-system-1}
and find $u_{i+1}^{l}$, for $l=0, 1, 2, 3$,
by  \eqref{eq-ODE-system-2}.

For simplicity, we assume
$$u_0^0=u_0^1=u_0^2=u_0^3=0,\quad w_0|_{z=z_0}=0.$$
Otherwise, we set
$$\tilde{u}^{0}=u^{0}-\varphi, \ \tilde{u}^{1}=u^{1}-\partial_z \varphi,\
\tilde{u}^{A}=u^{A}-\partial_A \varphi,\ \tilde{w}=w-\partial_\tau\psi.$$

For some $M>0$ and $\rho>0$, we define
$$F(s)=\frac{M}{1-\frac{s}{\rho}}.$$
We now consider a given point on $\Sigma_0$, say $(0,z_{*},0,0)$ with $z_*>0$.
Set
$$s=a^2\tau+a(z-z_{*})+x_2+x_3.$$
In a neighborhood of $(0,z_{*},0,0)$,
take $M>0, \rho>0$, and $a>1$, such that
the function
\begin{equation*}
F(s)
=\frac{M}{1-\rho^{-1}\big(a^2\tau+a(z-z_{*})+x_2+x_3\big)}
\end{equation*}
is a majorizing function of $g^{ij}$, $\bar{a}^{i}$, $c$, $f$ and $\varphi, \psi$. Then,
\begin{align}\label{eq-Majorizing-system}\begin{split}
\partial_\tau u^{i}&=F(s)\{\partial_z u^{1}
+\sum_{A,B=2,3}[\partial_z u^{A}+\partial_B u^{A}
+\partial_A w]+\sum_{j=0}^{3}u^{j}+w+1\},\\
\partial_z w&=F(s)\{\partial_z u^{1}
+\sum_{A,B=2,3}[\partial_z u^{A}+\partial_B u^{A}]+\sum\limits_{j=0}^{3}u^{j}+w+1\}
\end{split}\end{align}
forms a majorizing system. We now treat $s$ as an independent variable.
To construct a special solution $u^{l}=U(s), l=0,1,2,3$, and $w=W(s)$ of \eqref{eq-Majorizing-system},
we consider a system of linear ordinary differential equations given by
\begin{align}\label{eq-Majorizing-ODE}\begin{split}
\{a^2-F(s)(3a+4)\}\frac{d U}{d s}-2F(s)\frac{d W}{d s}&=F(s)(4U+W+1),\\
-F(s)(3a+4)\frac{d U}{d s}+a\frac{d W}{d s}&=F(s)(4U+W+1),
\end{split}\end{align}
with $U(0)=W(0)=0$. Take $\lambda$ small such that
\begin{equation*}
\begin{bmatrix}
a^2-F(s)(3a+4)&-2F(s)\\
-F(s)(3a+4)&a
\end{bmatrix}
\end{equation*}
is positive definite. Then, we can solve \eqref{eq-Majorizing-ODE}
and its solutions $U$ and $W$ are real analytic in the domain of $F(s)$.
Therefore, the domain where $u^{i}$ and $ w$ exist and are real analytic is the same as the domain
where all coefficients and initial values are real analytic. Similarly as in \cite{Duff1958},
by $a^2-F(s)(3a+4)>0$, $-2F(s)<0$, $-F(s)(3a+4)<0$, $a>0$,
the coefficients in the series of $U(s)$ and $W(s)$ are nonnegative, provided that $U(0)$ and $W(0)$ are $0$.
This proves the existence of an analytic solution $v$ of \eqref{eq-main-eq}.

We now consider the general case.
For every $m$,  we can find sequences of polynomials $P^{j}$ and $Q^{j}$ such that
\begin{equation*}
\lim\limits_{j\rightarrow\infty}D^{\alpha}P^{j}=D^{\alpha}\varphi\quad\text{uniformly on }\bar\Sigma_0,
\end{equation*}
for any $\alpha$ with $|\alpha|\leq 2m+5$ and $\alpha_0=0$, and
\begin{equation*}
\lim\limits_{j\rightarrow\infty}D^{\beta}Q^{j}=D^{\beta}\psi\quad\text{uniformly on }\Sigma_1,
\end{equation*}
for any $\beta$ with
$|\beta|\leq 2m+5$ with $\beta_1=0$. Denote by $v^{j}$ the solution of \eqref{eq-main-eq}
with the initial value and the boundary value given by $P^{j}$ and $Q^{j}$, respectively.
By the $H^{p}$-estimates provided by Corollary \ref{theorem-Hp-estimates} and the Sobolev embedding,
we find that the $v^{j}$ converges, as $j\rightarrow\infty$, to a solution
$v\in C^{m}(\Omega\cup\Sigma_0\cup \Sigma_1)$ of \eqref{eq-main-eq}
with the initial value and the boundary value given by $\varphi$ and $ \psi$, respectively.
\end{proof}

We now make an important remark concerning the regularity of solutions established in
Theorem \ref{theorem-Existence}.

\begin{remark}\label{remark-Regularity}
In Corollary \ref{theorem-Hp-estimates},  we  assume that solutions
are regular up to the boundary portion $\Sigma_0\cup \Sigma_1$.
We do not assume that solutions are regular up to $z=0$. This is because
estimates established in Corollary \ref{theorem-Hp-estimates} is applied to analytic solutions
as in the proof of Theorem \ref{theorem-Existence} and analytic solutions are proved to exist only
for $z>0$. It is not clear at first that the analytic solutions can be extended up to $z=0$.
With an approximation process as in the proof of Theorem \ref{theorem-Existence},
we can establish the existence of nonanalytic solutions. Since these solutions also satisfy
the estimates in Corollary \ref{theorem-Hp-estimates}, they are actually regular up to the entire
boundary of $\Omega$ and, in particular, the portion $z=0$, by the Sobolev embedding. \end{remark}

We now prove the following result.

\begin{corollary}\label{cor-Expansions}
For some integer $m\ge 2$,  assume $\varphi\in H^{2m+5}(\Sigma_0)$ and
$\psi\in H^{2m+5}(\Sigma_{1})$, with $\varphi=\psi$ on $\Sigma_0\cap \Sigma_1$.
Let $v$ be the solution as in Theorem \ref{theorem-Existence}. Then, for each $i=0, \cdots, m-1$,
there exists a function $v_i\in C^{m-i}([0, T]\times S^2)$, such that, for any
$(\tau, z, \theta)\in \Omega$,
$$\big|v(\tau, z, \theta)-\sum_{i=0}^{m-1}v_i(\tau, \theta)z^i\big|\le
Cz^m\big\{\|\varphi\|_{H^{2m+5}(\Sigma_0)}
+\|\psi\|_{H^{2m+5}(\Sigma_1)}\big\},$$
where $C$ is a positive constant depending on $m$, $\lambda$, $\Lambda$, $T$, $z_0$, and  $g^{ij}$.
\end{corollary}

\begin{proof} By Theorem \ref{theorem-Existence} and Remark \ref{remark-Regularity}, we have
$v\in C^m(\bar\Omega)$. Set
$$v_i=\frac{1}{i!}\partial_z^iv\big|_{z=0}.$$
We note that $v_i$ defined here differs from $v_i$ in Corollary \ref{cor-radiation-Hk-j1} by a constant multiple.
Then, for any
$(\tau, z, \theta)\in [0,T]\times (0,z_0]\times S^2$,
$$\big|v(\tau, z, \theta)-\sum_{i=0}^{m-1}v_i(\tau, \theta)z^i\big|\le
Cz^m\|\partial_z^mv\|_{L^\infty(\Omega)}.$$
We hence have the desired result.
\end{proof}


Next, we write the equation $Lu=0$ as
\begin{equation*}
2\partial_{z\tau}v+z^{2}V\partial_{zz}v+2g^{1A}\partial_{zA}v+g^{AB}\partial_{AB}v+a^{i}\partial_{i}v+bv=0,
\end{equation*}
where $i=1, 2, 3$.
Formally, we assume that $v$ has an asymptotic expansion given by
$$v(\tau, z, \theta)=\sum\limits_{k=0}^{\infty}v_{k}(\tau, \theta)z^{k}.$$
By $V\to1$ and $U^A\to0$ as $z\to 0$,
we have $a^{1}\to0$ as $z\to0$.
Suppose
\begin{align*}
V&=1+\sum_{k=1}^{\infty}V_{k}(\tau, \theta)z^{k},\\
g^{iA}&=\sum_{k=1}^{\infty}g_{k}^{iA}(\tau, \theta)z^{k}\quad\text{for }i=1, 2, 3\text{ and }A=2,3,\\
a^{i}&=\sum_{k=0}^{\infty}a_{k}^{i}(\tau, \theta)z^{k}\quad\text{for }i=1, 2, 3,\text{ and } a_{0}^{1}=0,\\
b&=\sum_{k=0}^{\infty}b_{k}(\tau, \theta)z^{k}.
\end{align*}
Then,
\begin{align*}
2\frac{\partial v_{1}}{\partial\tau}+g_{0}^{AB}\partial_{AB}v_{0}
+a_{0}^{A}\partial_A v_{0}+b_{0}v_{0}&=0,\\
4\frac{\partial v_{2}}{\partial\tau}+2g_{1}^{1A}\frac{\partial v_{1}}{\partial A}+a_{1}^{1}v_{1}
+\sum\limits_{i+j=1}[g_{i}^{AB}\partial_{AB}v_{j}
+a_{i}^{A}\partial_A v_{j}+b_{i}v_{j}]&=0,
\end{align*}
and, for $k\geq2$,
\begin{align*}
2(k+1)\frac{\partial v_{k+1}}{\partial\tau}+\sum\limits_{i+j=k,i\geq2}i(i-1)v_{i}V_{j}
+\sum\limits_{i+j-1=k; i,j\geq1}2ig_{j}^{1A}\partial_A v_{i}& \\
+\sum\limits_{i+j-1=k;i,j\geq1}a_{i}^{1}v_{j}+\sum\limits_{i+j=k}[g_{i}^{AB}\partial_{AB}v_{j}
+a_{i}^{A}\partial_A v_{j}+b_{i}v_{j}]&=0.
\end{align*}
Note $v_k=(\partial_z^kv/k!)|_{z=0}$. Hence, $v_k|_{\tau=0}=\lim_{z\to 0}(\partial_z^k\varphi/k!)$.
Therefore, with $v_{0}$ on $[0,T]\times S^2$ and the initial values $v_k$ on $\{0\}\times S^2$,
we can obtain $v_{k}$ on $[0,T]\times S^2$ successively, for $k\geq 1$.

To end this paper, we simply note that Theorem \ref{thrm-main}
follows easily from Theorem \ref{theorem-Existence} and Corollary \ref{cor-Expansions},
with the help of the change of coordinates in Section \ref{sec-Metrics}.

\end{document}